%% file: parab.tex
\documentclass[11pt]{article}
\usepackage{amsfonts}
\usepackage{amssymb,amsmath}
\usepackage{graphicx}
\usepackage{pictex}
\usepackage{multirow}
\usepackage[pdftex]{hyperref}

\newtheorem{theorem}{Theorem}[section]
\newtheorem{lemma}[theorem]{Lemma}
\newtheorem{remark}[theorem]{Remark}
\newtheorem{example}[theorem]{Example}
\newtheorem{corollary}[theorem]{Corollary}
\newenvironment{proof}[1][Proof]{\textbf{#1.} }{\ \rule{0.5em}{0.5em}}

\newenvironment{acknowledgement}{\smallskip{\sc Acknowledgments.}\rm}{\smallskip}
\newenvironment{notation}{\smallskip{\sc Notation.}\rm}{\smallskip}
\input{macro}

\begin{document}

\title{Heat kernel estimates on connected sums of parabolic manifolds}
\author{Alexander Grigor'yan\thanks{%
Partially supported by SFB 701 of German Research Council} \\
%EndAName
Department of Mathematics \\
University of Bielefeld\\
33501 Bielefeld, Germany \\
grigor@math.uni-bielefeld.de \and Satoshi Ishiwata\thanks{%
Partially supported by JSPS, KAKENHI 25800034} \\
%EndAName
Department of Mathematical Sciences\\
Yamagata University\\
Yamagata 990-8560, Japan\\
ishiwata@sci.kj.yamagata-u.ac.jp \and Laurent Saloff-Coste \thanks{%
Partially supported by NSF grant DMS--1404435} \\
%EndAName
Department of Mathematics\\
Cornell University\\
Ithaca, NY, 14853-4201, USA\\
lsc@math.cornell.edu}
\date{July 2016}
\maketitle

\begin{abstract}
We obtain matching two sided estimates of the heat kernel on a connected sum
of parabolic manifolds, each of them satisfying the Li-Yau estimate. 
%%%%%%%%%%%%%%%%%%%%%%%%%%%
%Except some cases like as $\mathbb{R}^2\# \mathbb{R}^2$, the bottleneck
%effect cannot be observed. 
%%%%%%%%%%%%%%%%%%%%%%%%%%%%%%%%%%
The key result is the on-diagonal upper bound of the heat kernel at a
central point. Contrary to the non-parabolic case (which was settled in \cite%
{G-SC ends}), the on-diagonal behavior of the heat kernel in our case is
determined by the end with the \textit{maximal} volume growth function. As
examples, we give explicit heat kernel bounds on the connected sums $\mathbb{%
R}^{2}\#\mathbb{R}^{2}$ and $\mathcal{R}^{1}\#\mathbb{R}^{2}$ where $%
\mathcal{R}^{1}=\mathbb{R}_{+}\times \mathbb{S}^{1}.$
\end{abstract}

\tableofcontents

%%%%%%%%%%%%%%document start %%%%%%%%%%%%%%%

\section{Introduction}

\label{Introduction}

Let $M$ be a Riemannian manifold. The heat kernel $p(t,x,y)$ on $M$ is the
minimal positive fundamental solution of the heat equation $\partial
_{t}u=\Delta u$ on $M$ where $u=u\left( t,x\right) $, $t>0$, $x\in M$ and $%
\Delta $ is the (negative definite) Laplace-Beltrami operator on $M$. For
example, in $\mathbb{R}^{n}$ the heat kernel is given by the classical
Gauss-Weierstrass formula%
\begin{equation*}
p(t,x,y)=\frac{1}{(4\pi t)^{n/2}}\exp \left( -\frac{|x-y|^{2}}{4t}\right) .
\end{equation*}%
The heat kernel is sensitive to the geometry of the underlying manifold $M$,
which results in numerous applications of this notion in differential
geometry. On the other hand, the heat kernel has a probabilistic meaning: $%
p(t,x,y)$ is the transition density of Brownian motion $(\{X_{t}\}_{t\geq0},%
\{\mathbb{P}_{x}\}_{x \in M})$ on $M$. Namely, for any Borel set $A\subset M$%
, we have%
\begin{equation*}
\mathbb{P}_{x}(X_{t}\in A)=\int_{A}p(t,x,y)dy,
\end{equation*}%
where $\mathbb{P}_{x}(X_{t}\in A)$ is the probability that Brownian particle
starting at the point $x$ will be found in the set $A$ in time $t$.

From now on let us assume that the manifold $M$ is non-compact and
geodesically complete. Dependence of the long time behavior of the heat
kernel on the large scale geometry of $M$ is an interesting and important
problem that has been intensively studied during the past few decades by
many authors (see, for example, \cite{Davies CTM}, \cite{G AMS}, \cite{SC
LNS} and references therein). In the case when the Ricci curvature of $M$ is
non-negative, P.Li and S.-T.Yau proved in their pioneering work \cite{Li-Yau}
the following estimate, for all $x,y\in M$ and $t>0$:%
\begin{equation}
p(t,x,y)\asymp \frac{C}{V(x,\sqrt{t})}\exp \left( -b\frac{d^{2}(x,y)}{t}%
\right) ,  \tag{$LY$}  \label{LY type}
\end{equation}%
where the sign $\asymp $ means that both $\leq $ and $\geq $ hold but with
different values of positive constants $C$ and $b$, $V(x,r)$ is the
Riemannian volume of the geodesic ball of radius $r$ centered at $x\in M$,
and $d\left( x,y\right) $ is the geodesic distance between the points $x,y$.

The estimate $($\ref{LY type}$)$ is satisfied also for the heat kernel of
uniformly elliptic operators in divergence form in $\mathbb{R}^{n}$ as was
proved by Aronson \cite{Aronson}. It was proved by Fabes and Stroock \cite%
{Fabes-Stroock}, that the estimate $($\ref{LY type}$)$ is equivalent to the
uniform parabolic Harnack inequality (see also \cite{SC LNS}). Grigor'yan 
\cite{Grigoryan 1991} and Saloff-Coste \cite{SC 1992}, \cite{SC LNS} proved
that $($\ref{LY type}$)$ is equivalent to the conjunction of the Poincar\'{e}
inequality and the volume doubling property.

One of the simplest example of a manifold where $($\ref{LY type}$)$ fails is
the hyperbolic space $\mathbb{H}^{n}.$ A more interesting counterexample was
constructed by Kuz'menko and Molchanov \cite{Kuz'menko-Molchanov}: they
showed that the connected sum $\mathbb{R}^{n}\#\mathbb{R}^{n}$ of two copies
of $\mathbb{R}^{n}$, $n\geq 3$, admits a non-trivial bounded harmonic
function, which implies that the Harnack inequality and, hence, $($\ref{LY
type}$)$ cannot be true. Benjamini, Chavel and Feldman \cite%
{Benjamini-Chavel-Feldman} explained this phenomenon by a bottleneck-effect:
if $x$ and $y$ belong to the different ends of the manifold $\mathbb{R}^{n}\#%
\mathbb{R}^{n}$ and $\left\vert x\right\vert \approx \left\vert y\right\vert
\approx \sqrt{t}\rightarrow \infty $ then $p\left( t,x,y\right) \ll t^{-n/2}$
where $t^{-n/2}$ is predicted by the right hand side of $($\ref{LY type}$)$.
This phenomenon is especially transparent from probabilistic viewpoint:
Brownian particle can go from $x$ to $y$ only through the central part,
which reduces drastically the transition density (see Fig. \ref{rn+rn}). A
similar phenomenon was observed by B.Davies \cite{Davies 97} on a model case
of one-dimensional line complex.

\FRAME{ftbphFU}{2.8625in}{1.5186in}{0pt}{\Qcb{Brownian path goes from $x$ to 
$y$ via the bottleneck}}{\Qlb{rn+rn}}{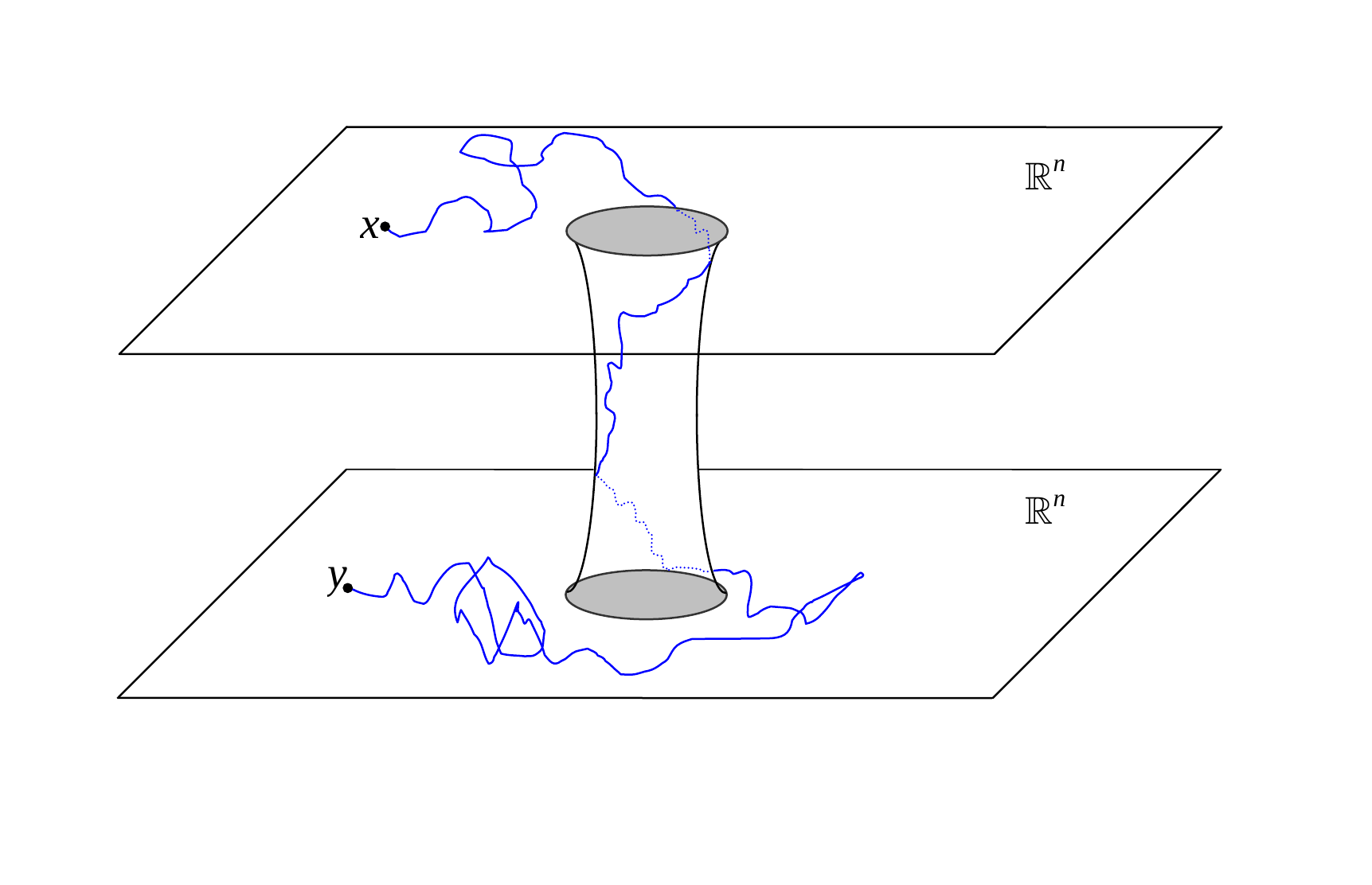}{\special{language
"Scientific Word";type "GRAPHIC";maintain-aspect-ratio TRUE;display
"USEDEF";valid_file "F";width 2.8625in;height 1.5186in;depth
0pt;original-width 6.3183in;original-height 3.3356in;cropleft "0";croptop
"1";cropright "1";cropbottom "0";filename 'rn+rn.eps';file-properties
"XNPEU";}}

Based on these early works, the first and the third authors of the present
paper started a project on heat kernel bounds on connected sums of
manifolds, provided each of them satisfies the Li-Yau estimate $($\ref{LY
type}$)$. The results of this study are published in a series \cite{G-SC
letter}, \cite{G-SC Dirichlet}, \cite{G-SC hitting}, \cite{G-SC ends}, and 
\cite{G-SC FK}. In particular, they obtained in \cite{G-SC ends} matching
upper and lower estimates of heat kernels on connected sums of manifolds
when at least one of them is \textit{non-parabolic}. Recall that a manifold $%
M$ called \textit{parabolic} if Brownian motion on $M$ is recurrent, and 
\textit{non-parabolic} otherwise. There are several equivalent definitions
of parabolicity in different terms (see, for example, \cite{G 1999}).

In this paper we complement the results \cite{G-SC ends} by proving
two-sided estimates of heat kernels on connected sums of \textit{parabolic}
manifolds. The detailed statements are given in the next section. We
illustrate our results on the following two examples.

Consider first the manifold $M=\mathcal{R}^{1}\#\mathbb{R}^{2}$, where $%
\mathcal{R}^{1}=\mathbb{R}_{+}\times \mathbb{S}^{1}$ (see Fig. \ref{r1+r2}%
). \FRAME{ftbphFU}{2.8115in}{1.6604in}{0pt}{\Qcb{Connected sum $\mathcal{R}%
^{1}\#\mathbb{R}^{2}$}}{\Qlb{r1+r2}}{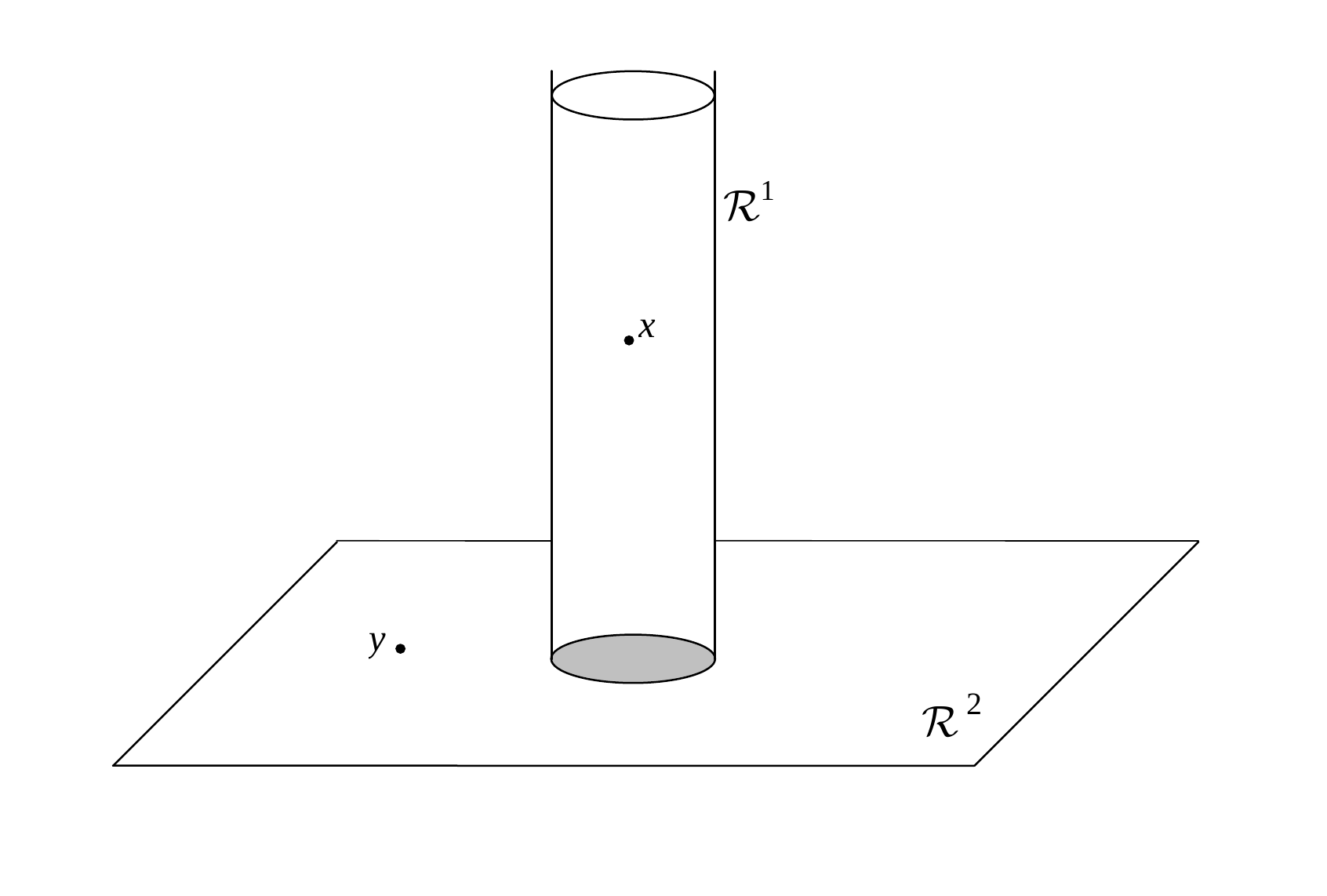}{\special{language
"Scientific Word";type "GRAPHIC";maintain-aspect-ratio TRUE;display
"USEDEF";valid_file "F";width 2.8115in;height 1.6604in;depth
0pt;original-width 6.0026in;original-height 3.535in;cropleft "0";croptop
"1";cropright "1";cropbottom "0";filename 'jpic8.eps';file-properties
"XNPEU";}} For $x\in M$, define $|x|:=d(x,K)+e$, where $K\subset M$ is the
central part of $M$. Then we obtain that for $x\in \mathcal{R}^{1}$, $y\in 
\mathbb{R}^{2}$ and $t>1$ 
\begin{equation*}
p(t,x,y)\asymp \left\{ 
\begin{array}{ll}
\frac{1}{t}e^{-b\frac{d^{2}(x,y)}{t}} & \mbox{if }\left\vert y\right\vert >%
\sqrt{t}, \\ 
\frac{1}{t}\left( 1+\frac{|x|}{\sqrt{t}}\log \frac{e\sqrt{t}}{|y|}\right)  & %
\mbox{if }\left\vert x\right\vert ,\left\vert y\right\vert \leq \sqrt{t}, \\ 
\frac{1}{t}\log \frac{e\sqrt{t}}{|y|} & \mbox{if }\left\vert x\right\vert >%
\sqrt{t}\geq \left\vert y\right\vert .%
\end{array}%
\right. 
\end{equation*}%
In particular, if $|x|$, $|y|$ are bounded and $t\rightarrow \infty $, then 
\begin{equation*}
p(t,x,y)\approx \frac{1}{t}.
\end{equation*}%
If $\left\vert x\right\vert \approx \sqrt{t}\rightarrow \infty $ and $%
\left\vert y\right\vert $ remains bounded, then 
\begin{equation*}
p(t,x,y)\approx \frac{\log t}{t}.
\end{equation*}

Consider now the manifold $M=\mathbb{R}^{2}\#\mathbb{R}^{2}$, or,
equivalently, a catenoid (see Fig. \ref{catenoid}).   \FRAME{ftbphFU}{%
2.1586in}{1.5134in}{0pt}{\Qcb{Catenoid}}{\Qlb{catenoid}}{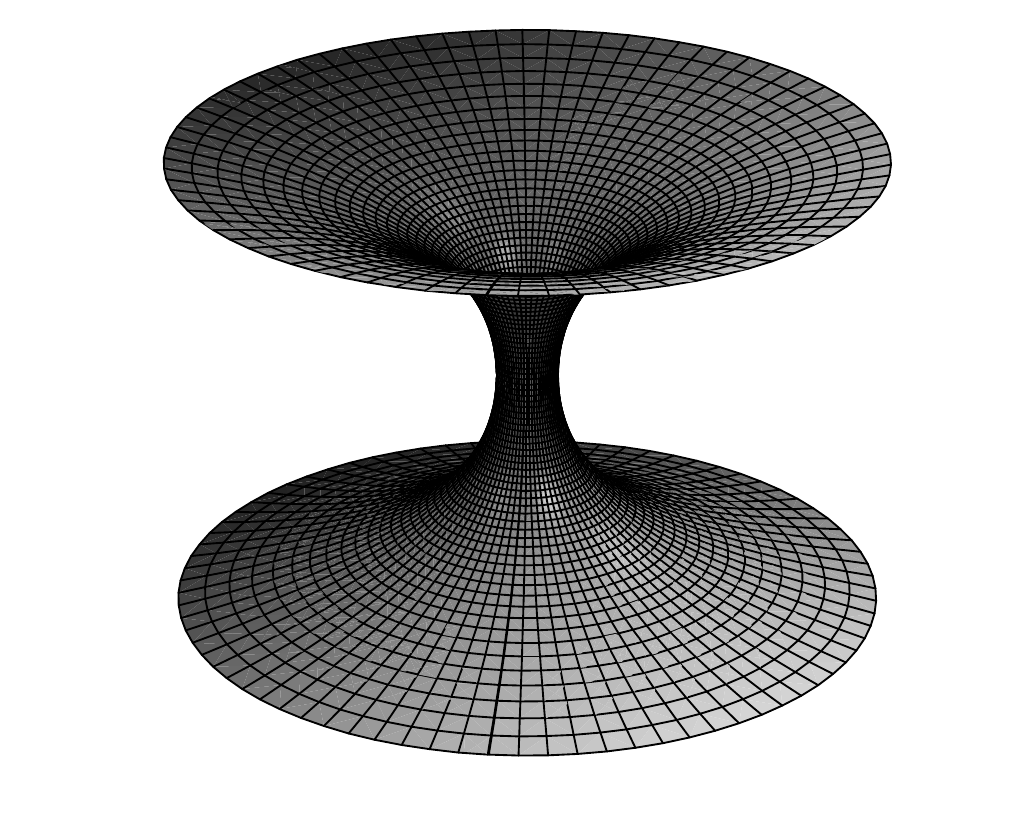}{%
\special{language "Scientific Word";type "GRAPHIC";maintain-aspect-ratio
TRUE;display "USEDEF";valid_file "F";width 2.1586in;height 1.5134in;depth
0pt;original-width 4.6428in;original-height 3.2467in;cropleft "0";croptop
"1";cropright "1";cropbottom "0";filename 'catenoid.eps';file-properties
"XNPEU";}}

Then we have the following estimate for all $x,y$ lying in different sheets
and for $t>1$: 
\begin{equation*}
p(t,x,y)\asymp \left\{ 
\begin{array}{ll}
\frac{1}{t\log ^{2}t}\left( \log t+\log ^{2}\sqrt{t}-\log |x|\log |y|\right) 
& \mbox{if }|x|,|y|\leq \sqrt{t}, \\ 
\frac{1}{t\log t}\log \frac{e\sqrt{t}}{|y|}e^{-b\frac{d^{2}(x,y)}{t}} & %
\mbox{if }|y|\leq \sqrt{t}<|x|, \\ 
\frac{1}{t\log t}\log \frac{e\sqrt{t}}{|x|}e^{-b\frac{d^{2}(x,y)}{t}} & %
\mbox{if }|x|\leq \sqrt{t}<|y|, \\ 
\frac{1}{t}\left( \frac{1}{\log |x|}+\frac{1}{\log |y|}\right) e^{-b\frac{%
d^{2}(x,y)}{t}} & \mbox{if }|x|,|y|>\sqrt{t}.%
\end{array}%
\right. 
\end{equation*}

In particular, if $\vert x \vert$, $\vert y \vert$ are bounded and $t
\rightarrow \infty$, then 
\begin{equation*}
p(t,x,y) \approx \frac{1}{t}.
\end{equation*}

If $|x|\approx |y|\approx \sqrt{t}\rightarrow \infty $ then 
\begin{equation*}
p(t,x,y)\approx \frac{1}{t\log t}.
\end{equation*}%
The heat kernel estimates on $\mathbb{R}^{2}\#\mathbb{R}^{2}$ was also
obtained in \cite{G-SC ends} by an ad hoc method. In the present paper these
estimates are part of our general Theorem \ref{T1}. We also give further
examples, in particular, the heat kernel estimates on $\mathcal{R}^{1}\#%
\mathcal{R}^{1}\#\mathbb{R}^{2}.$ 

In the next section we introduce necessary definitions and state our main
results. In Section \ref{SecGeneral} we prove some auxiliary results about
the integrated resolvent. In Section \ref{SecOndiag} we prove the main
technical result of this paper -- Theorem \ref{main theorem} about
on-diagonal upper bound of the heat kernel on the connected sum of parabolic
manifolds. Finally, in Section \ref{SecOff} we use Theorem \ref{main theorem}
and the gluing techniques from \cite{G-SC ends} to obtain full off-diagonal
estimates of the heat kernels; they are stated in Theorems \ref{T1}-\ref{T3}
and Corollaries \ref{power} and \ref{Corcases}.

\begin{notation}
Throughout this article, the letters $c,C,b,...$ denote positive constants
whose values may be different at different instances. When the value of a
constant is significant, it will be explicitly stated. The notation $%
f\approx g$ for two non-negative functions $f,g$ means that there are two
positive constants $c_{1},c_{2}$ such that $c_{1}g$ $\leq f\leq c_{2}g$ for
the specified range of the arguments of $\,f$ and $g$.
\end{notation}

%%%%%%%%%%%%%%%%%%%%%%%%%%%%%%%%  Statement of main result   %%%%%%%%

\section{Statement of main results and examples}

\label{section 2}

\setcounter{equation}{0}The main result will be stated in a more general
setting of weighted manifolds that is explained below.

\subsection{Weighted manifolds}

Let $M$ be a connected Riemannian manifold of dimension $N$. The Riemannian
metric of $M$ induces the geodesic distance $d(x,y)$ between points $x,y\in
M $ and the Riemannian measure $d\mathrm{vol}.$ Given a smooth positive
function $\sigma $ on $M$, let $\mu $ be the measure on $M$ given by $d\mu
(x)=\sigma (x)d\mathrm{vol}(x)$. The pair $(M,\mu )$ is called a \textit{%
weighted manifold}. Any Riemannian manifold can be considered also as a
weighted manifold with $\sigma \equiv 1$.

The Laplace operator $\Delta $ of the weighted manifold $\left( M,\mu
\right) $ is defined by%
\begin{equation*}
\Delta =\frac{1}{\sigma }\func{div}\left( \sigma \nabla \right) ,
\end{equation*}%
where $\func{div}$ and $\nabla $ are the divergence and the gradient of the
Riemannian metric of $M$. It is easy to see that $\Delta $ is the generator
of the following Dirichlet form%
\begin{equation*}
D\left( f,f\right) =\int_{M}\left\vert \nabla f\right\vert ^{2}d\mu
\end{equation*}%
in $W^{1,2}\left( M,\mu \right) $. The associated heat semigroup $e^{t\Delta
}$ has always a smooth positive kernel $p\left( t,x,y\right) $ that is
called the heat kernel of $\left( M,\mu \right) $. At the same time, $%
p\left( t,x,y\right) $ is the minimal positive fundamental solution of the
corresponding heat equation $\partial _{t}u=\Delta u$ on $M\times \mathbb{R}%
_{+}$ (see \cite{G AMS}). The heat kernel is also the transition probability
density of Brownian motion $\left( \left\{ X_{t}\right\} ,\left\{ \mathbb{P}%
_{x}\right\} \right) $ on $M$ that is generated by $\Delta $.

A weighted manifold $\left( M,\mu \right) $ is called \textit{parabolic} if
any positive superharmonic function on $M$ is constant, and \textit{%
non-parabolic} otherwise. The parabolicity is equivalent to each of the
following properties, that can be regarded as equivalent definitions (see,
for example, \cite{G 1999}):

\begin{enumerate}
\item There exists no positive fundamental solution of $-\Delta .$

\item $\int^{\infty }p\left( t,x,y\right) dt=\infty $ for all/some $x,y\in M$%
.

\item Brownian motion on $M$ is recurrent.
\end{enumerate}

\subsection{Notion of connected sum}

\label{notion}

Let $(M,\mu )$ be a geodesically complete non-compact weighted manifold. Let 
$K\subset M$ be a connected compact subset of $M$ with non-empty interior
and smooth boundary such that $M\setminus K$ has $k$ non-compact connected
components $E_{1},\ldots ,E_{k}$; moreover, assume also that the closures $%
\overline{E}_{i}$ are disjoint. We refer to each $E_{i}$ as an \textit{end}
of $M$. Clearly, $\partial K$ is a disjoint union of $\partial E_{i}$, $%
i=1,...,k$.

Assume also that $E_{i}$ is isometric to the exterior of a compact set $%
K_{i} $ in another weighted manifold $(M_{i},\mu _{i})$. Then we refer to $M$
as the connected sum of $M_{1},...,M_{k}$ and write 
\begin{equation*}
M=M_{1}\#M_{2}\#\cdots \#M_{k}
\end{equation*}%
(see Fig. \ref{figure: connectedsum1}). 
\begin{figure}[tbph]
\begin{center}
\scalebox{0.9}{
\includegraphics{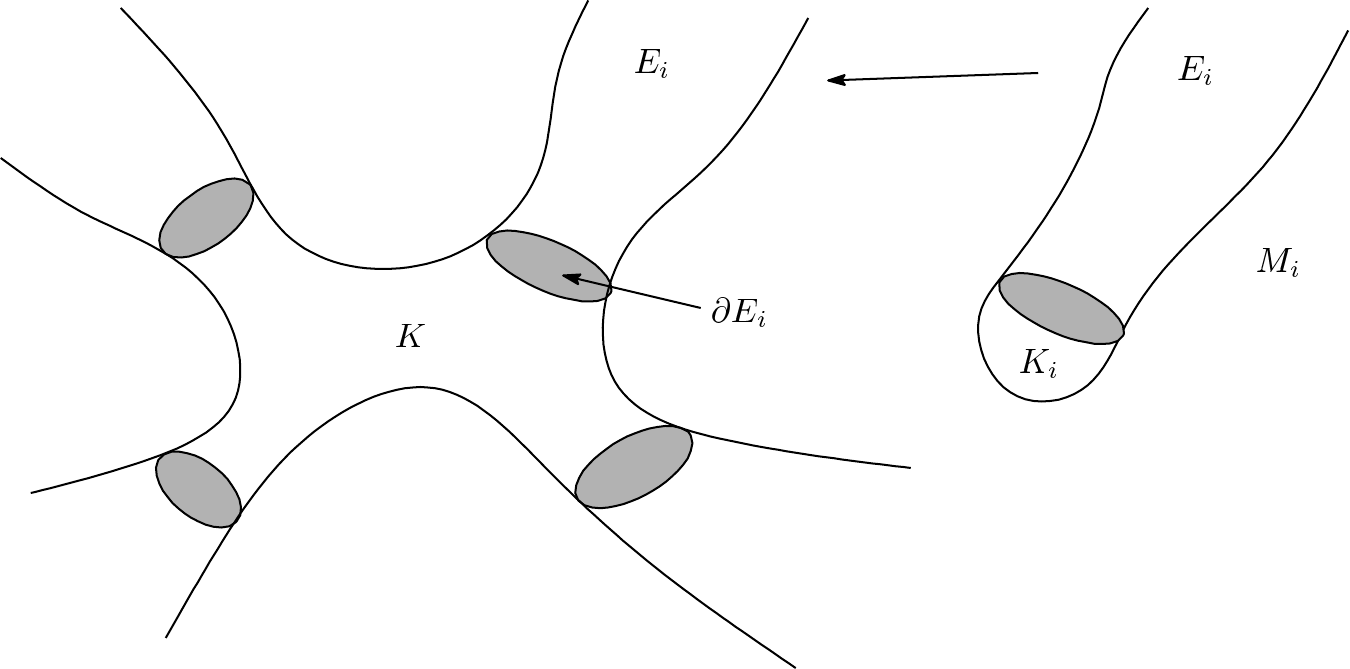} 
}
\end{center}
\caption{Connected sum $M=M_{1}\#M_{2}\cdots \#M_{k}$.}
\label{figure: connectedsum1}
\end{figure}

Denote by $d_{i}$ the geodesic distance on $M_{i}$ and by $B_{i}\left(
x,r\right) $ the geodesic ball in $M_{i}$ of radius $r$ centered at $x\in
M_{i}$. Set also $V_{i}\left( x,r\right) =\mu _{i}\left( B_{i}\left(
x,r\right) \right) $. Fix a reference point $o_{i}\in K_{i}$ and set 
\begin{equation*}
V_{i}(r)=V_{i}(o_{i},r).
\end{equation*}%
In this paper we always assume that every manifold $M_{i}$, $i=1,\ldots ,k$,
satisfies the following four conditions.

\begin{enumerate}
\item[$\left( a\right) $] The heat kernel $p_{i}\left( t,x,y\right) $ of $%
\left( M_{i},\mu _{i}\right) $ satisfies the Li-Yau estimate $($\ref{LY type}%
$)$, that is, 
\begin{equation}
p_{i}\left( t,x,y\right) \asymp \frac{C}{V_{i}\left( x,\sqrt{t}\right) }\exp
\left( -b\frac{d_{i}^{2}\left( x,y\right) }{t}\right) .  \label{LYi}
\end{equation}

\item[$\left( b\right) $] $M_{i}$ is parabolic; under the standing
assumption (\ref{LYi}), the parabolicity of $M_{i}$ is equivalent to 
\begin{equation}
\int^{\infty }\frac{rdr}{V_{i}\left( r\right) }=\infty .  \label{Vpar}
\end{equation}

\item[$\left( c\right) $] $M_{i}$ has \textit{relatively connected annuli},
that is, there exists a positive constant $A>1$ such that for any $r>A^{2}$
and all $x,y\in M_{i}$ with $d_{i}(o_{i},x)=d_{i}(o_{i},y)=r$, there exists
a continuous path from $x$ to $y$ staying in $B_{i}(o_{i},Ar)\setminus
B_{i}(o,A^{-1}r)$. We denote this condition shortly by $\left( RCA\right) $.

\item[$\left( d\right) $] $M_{i}$ is either \textit{critical} or \textit{%
subcritical}; here $M_{i}$ is called critical if, for all large enough $r$,%
\begin{equation*}
V_{i}(r)\approx r^{2},
\end{equation*}%
and subcritical if, for all large enough $r$, 
\begin{equation}
\int_{1}^{r}\frac{sds}{V_{i}(s)}\leq \frac{Cr^{2}}{V_{i}(r)}.
\label{subcritical}
\end{equation}
\end{enumerate}

For example, if $V_{i}(r)\approx r^{\alpha }\log ^{\beta }r$ for some $%
0<\alpha <2$ and $\beta \in \mathbb{R}$, then $M_{i}$ is subcritical. On the
other hand, in the case $V_{i}\left( r\right) \approx \frac{r^{2}}{\log
^{\beta }r}$ with $\beta >0$ the manifold $M_{i}$ is neither critical nor
subcritical, although still parabolic.

Let us describe a class of manifolds satisfying all the hypotheses $\left(
a\right) -\left( d\right) $. For any $0<\alpha \leq 2$ consider a Riemannian 
\textit{model manifold} $\mathcal{R}^{\alpha }:=(\mathbb{R}^{2},g_{\alpha })$%
, where $g_{\alpha }$ is a Riemannian metric on $\mathbb{R}^{2}$ such that,
in the polar coordinates $\left( \rho ,\theta \right) $, it is given for $%
\rho >1$ by 
\begin{equation*}
g_{\alpha }=d\rho ^{2}+\rho ^{2(\alpha -1)}d\theta ^{2}.
\end{equation*}%
For example, if $\alpha =2$ then $g_{2}$ can be taken to be the Euclidean
metric of $\mathbb{R}^{2}$ so that in this case $\mathcal{R}^{2}=\mathbb{R}%
^{2}$. If $\alpha =1$ then $g_{1}=d\rho ^{2}+d\theta ^{2}$ so that the
exterior domain $\left\{ \rho >1\right\} $ of $\mathcal{R}^{1}$ is isometric
to the cylinder $\mathbb{R}_{+}\times \mathbb{S}$ (see Fig. \ref{pic5}).

\FRAME{ftbphFU}{6.3683in}{0.7923in}{0pt}{\Qcb{Model manifold $\mathcal{R}%
^{1} $}}{\Qlb{pic5}}{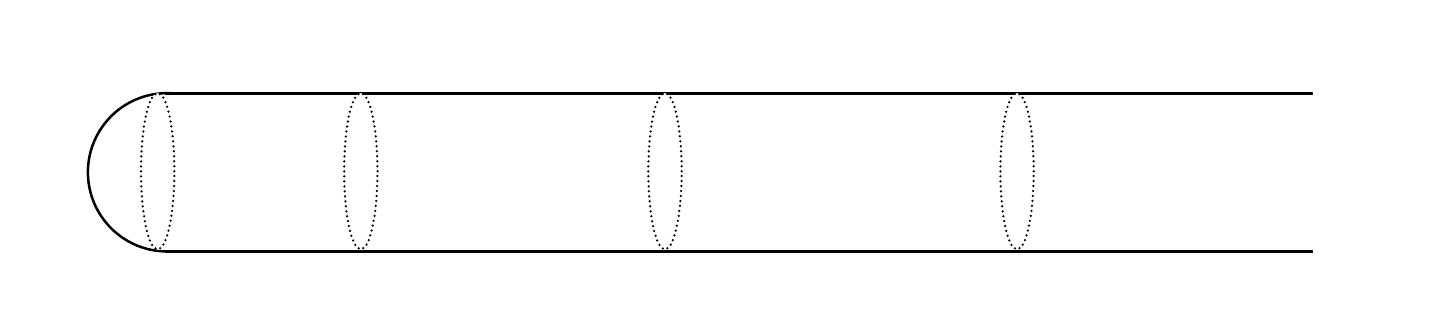}{\special{language "Scientific Word";type
"GRAPHIC";maintain-aspect-ratio TRUE;display "USEDEF";valid_file "F";width
6.3683in;height 0.7923in;depth 0pt;original-width 6.0026in;original-height
0.7231in;cropleft "0";croptop "1";cropright "1";cropbottom "0";filename
'PIC5.eps';file-properties "XNPEU";}}

For a general $0<\alpha <2$, the exterior domain $\left\{ \rho >1\right\} $
of $\mathcal{R}^{\alpha }$ is isometric to a certain surface of revolution
in $\mathbb{R}^{3}$.

Observe that the volume function $V(x,r)$ on $\mathcal{R}^{\alpha }$ admits
for $r>1$ the estimate%
\begin{equation}
V(x,r)\approx \left\{ 
\begin{array}{ll}
r^{\alpha }, & \left\vert x\right\vert <r \\ 
\min \left( r^{2},r\left\vert x\right\vert ^{\alpha -1}\right) , & 
\left\vert x\right\vert \geq r%
\end{array}%
\right. \approx \frac{r^{2}}{1+\frac{r}{(\left\vert x\right\vert +r)^{\alpha
-1}}}  \label{model}
\end{equation}%
(see \cite[Sec. 4.4]{G-SC stability}). In particular, if $x=o$, where $o$ is
the origin of $\mathbb{R}^{2}$, then 
\begin{equation}
V\left( o,r\right) \approx r^{\alpha }.  \label{Va}
\end{equation}%
By \cite[Prop. 4.10]{G-SC stability}, $\mathcal{R}^{\alpha }$ satisfies the
parabolic Harnack inequality and, hence, the Li-Yau estimate $($\ref{LY type}%
$)$. Obviously, $\mathcal{R}^{\alpha }$ satisfies (\ref{Vpar}) and, hence, $%
\mathcal{R}^{\alpha }$ is parabolic. It is easy to see that $\mathcal{R}%
^{\alpha }$ satisfies $\left( RCA\right) $. Note also that $\mathcal{R}%
^{\alpha }$ is critical if $\alpha =2$ and subcritical if $\alpha <2$.
Hence, $\mathcal{R}^{\alpha }$ satisfies all hypotheses $\left( a\right)
-\left( d\right) $.

One can make a similar family of examples also in class of weighted
manifolds. Indeed, for any $\alpha >0$ consider in $\mathbb{R}^{2}$ the
following measure

\begin{equation*}
d\mu _{\alpha }=\left( 1+\left\vert x\right\vert ^{2}\right) ^{\frac{\alpha 
}{2}-1}dx.
\end{equation*}%
It is easy to see that $\left( \mathbb{R}^{2},\mu _{\alpha }\right) $
satisfies (\ref{Va}). The Li-Yau estimate on $\left( \mathbb{R}^{2},\mu
_{\alpha }\right) $ holds by \cite[Prop. 4.9]{G-SC stability}. Hence, $%
\left( \mathbb{R}^{2},\mu _{\alpha }\right) $ satisfies all the hypotheses $%
\left( a\right) -\left( d\right) $ provided $0<\alpha \leq 2$.

Returning to the general setting, let us mention that the hypotheses $\left(
a\right) ,\left( b\right) ,\left( c\right) $ are essential for our main
result, whereas $\left( d\right) $ is technical. Probably, the method of
proof will work also without assuming $\left( d\right) $ but, even if that
is the case, the necessary computations will become much more technical and
complicated. So, we prefer to impose here the additional condition $\left(
d\right) $ to simplify the computational part of the proof, which even under 
$\left( d\right) $ remains quite involved.

Observe also that the condition $\left( b\right) $ follows from $\left(
d\right) $. Indeed, if the integral (\ref{Vpar}) converges then by (\ref%
{subcritical}) $V_{i}\left( r\right) \leq Cr^{2}$, which implies the
divergence of the integral in (\ref{Vpar}). However, for the aforementioned
reason, we state $\left( b\right) $ independently of $\left( d\right) $.

In fact, in the subcritical case we have%
\begin{equation}
V_{i}\left( r\right) =o\left( r^{2}\right) \ \ \text{as }r\rightarrow \infty
,  \label{or2}
\end{equation}%
as it follows from (\ref{Vpar}) and (\ref{subcritical}). Moreover,
substituting (\ref{or2}) to the left hand side of (\ref{subcritical}), we
obtain that, in the subcritical case, 
\begin{equation}
V_{i}\left( r\right) =o\left( \frac{r^{2}}{\log r}\right) \ \text{as }%
r\rightarrow \infty .  \label{or2logr}
\end{equation}

\subsection{On-diagonal estimates}

Denote by $d\left( x,y\right) $ the geodesic distance between points $x,y\in
M$ and by $V\left( x,r\right) $ the Riemannian volume of the geodesic ball
on $M$ of radius $r$ centered at $x\in M$. Fix a reference point $o\in K$
and set $V(r)=V(o,r)$. Set also 
\begin{equation*}
V_{\max }(r)=\max_{1\leq i\leq k}V_{i}(r).
\end{equation*}%
It is easy to see that, for all $r>0$, 
\begin{equation*}
V(r)\approx V_{1}(r)+V_{2}(r)+\cdots +V_{k}(r)\approx V_{\max }(r).
\end{equation*}%
The first main result of this paper is as follows.

\begin{theorem}
\label{main theorem} Let $M=M_{1}\#\cdots \#M_{k}$ be a connected sum of
non-compact complete manifolds $M_{1},\ldots ,M_{k}$. Assume that each $%
M_{i} $ is parabolic and satisfies {$($\ref{LY type}$)$} and $\left(
RCA\right) $. We also assume that each $M_{i}$ is either critical or
subcritical. Then we have 
\begin{equation}
p(t,o,o)\approx \frac{1}{V_{\max }(\sqrt{t})}\approx \frac{1}{V(\sqrt{t})},
\label{ptoo}
\end{equation}%
for all $t>0$.
\end{theorem}

Let us mention for comparison the following result of \cite{G-SC ends}: if
all manifolds $M_{i}$ are non-parabolic and satisfy $($\ref{LY type}$)$ and $%
\left( RCA\right) $, then the heat kernel on $M=$ $M_{1}\#\cdots \#M_{k}$
satisfies 
\begin{equation}
p(t,o,o)\approx \frac{1}{V_{\min }(\sqrt{t})},  \label{non-parabolic}
\end{equation}%
where 
\begin{equation*}
V_{\min }(r):=\min_{1\leq i\leq k}V_{i}(r).
\end{equation*}%
The proof of the upper bound in (\ref{non-parabolic}), that is, of the
inequality%
\begin{equation}
p\left( t,o,o\right) \leq \frac{C}{V_{\min }\left( \sqrt{t}\right) },
\label{ptmin}
\end{equation}%
goes as follows. By \cite[Prop. 5.2]{G Revista}, the upper bound in $($\ref%
{LY type}$)$ on $M_{i}$ is equivalent to a certain \textit{Faber-Krahn type}
inequality on $M_{i}$. Using a technique for merging of such inequalities,
developed in \cite[Thm. 3.5]{G-SC FK}, one obtains a similar Faber-Krahn
inequality on $M$, which then implies the heat kernel upper bound (\ref%
{ptmin}) by \cite[Thm. 5.2]{G Revista} (see \cite[Thm. 4.5]{G-SC FK} and 
\cite[Cor. 4.7]{G-SC ends} for the details). The reason for appearing of $%
V_{\min }$ in (\ref{ptmin}) is that the Faber-Krahn inequality on $M$ cannot
be stronger than that of each end $M_{i}$ and, hence, is determined by the
end with the smallest function $V_{i}\left( r\right) $.

The proof of the lower bound in (\ref{non-parabolic}), that is, of the
inequality 
\begin{equation}
p\left( t,o,o\right) \geq \frac{c}{V_{\min }\left( \sqrt{t}\right) }
\label{ptlow}
\end{equation}%
uses the comparison 
\begin{equation*}
p(t,x,y)\geq p_{E_{i}}(t,x,y)
\end{equation*}%
on each end $E_{i}$, where $p_{E_{i}}(t,x,y)$ is the Dirichlet heat kernel
on $E_{i}$ vanishing on $\partial E_{i}$. By \cite[Thm 3.1]{G-SC Dirichlet},
non-parabolicity of $M_{i}$ and $($\ref{LY type}$)$ imply that, away from $%
\partial E_{i}$,%
\begin{equation}
p_{E_{i}}(t,x,y)\geq cp_{i}\left( Ct,x,y\right) .  \label{pE}
\end{equation}%
It follows that, for any $i=1,...,k$, 
\begin{equation*}
p\left( t,o,o\right) \geq \frac{c}{V_{i}\left( \sqrt{t}\right) },
\end{equation*}%
which is equivalent to (\ref{ptlow}).

In the present setting, when all the manifolds $M_{i}$ are parabolic, both
arguments described above work but give non-optimal results. For example,
one obtains as above the upper bound (\ref{ptmin}), which in general is
weaker than the upper in (\ref{ptoo}). As far as the lower bound is
concerned, the estimate (\ref{pE}) fails in the parabolic case and has to be
replaced by a weaker one (cf. \cite[Thm 4.9]{G-SC Dirichlet}), which does
not yield an optimal lower bound for $p\left( t,o,o\right) .$ This explains
why we have to develop entirely new method for obtaining optimal bounds for $%
p\left( t,o,o\right) $ in the case when all manifolds $M_{i}$ are parabolic.
The most significant part of the estimate (\ref{ptoo}) is the upper bound%
\begin{equation}
p\left( t,o,o\right) \leq \frac{C}{V_{\max }\left( \sqrt{t}\right) }.
\label{pt<}
\end{equation}%
The proof of (\ref{pt<}) is the main achievement of the present paper. We
use for that a new method involving the \textit{integrated resolvent}%
\begin{equation*}
\gamma _{\lambda }\left( x\right) =\int_{K}\int_{0}^{\infty }e^{-t\lambda
}p\left( t,x,y\right) dtd\mu \left( y\right)
\end{equation*}%
defined for $\lambda >0.$ The parabolicity of $M$ implies that $\gamma
_{\lambda }\left( x\right) \rightarrow \infty $ as $\lambda \rightarrow 0,$
and the rate of increase of $\gamma _{\lambda }\left( x\right) $ as $\lambda
\rightarrow 0$ is related to the rate of decay of $p\left( t,o,o\right) $ as 
$t\rightarrow \infty .$ In fact, the integrated resolvent $\gamma _{\lambda
} $ on the connected sum $M$ satisfies a certain integral equation involving
as coefficients the Laplace transforms of the exit probabilities at each
end. This allows to estimate the rate of growth of $\gamma _{\lambda }$ as $%
\lambda \rightarrow 0$ and then to recover the upper bound (\ref{pt<}) in
the subcritical case. In the critical case one has to use instead $\partial
_{\lambda }\gamma _{\lambda }$.

Since $V_{\max }\left( r\right) \approx V\left( o,r\right) $ and $V\left(
o,r\right) $ satisfies the volume doubling property, the upper bound (\ref%
{pt<}) implies automatically a matching lower bound of $p\left( t,o,o\right) 
$ by \cite[Thm. 7.2]{Coulhon-Grigoryan} (see Section \ref{lower bound} for
the details).

\begin{remark}
\RM Kasahara and Kotani recently obtained in \cite[Example 6.1]{Kasahara}
the same on-diagonal heat kernel estimates for a connected sum of two Bessel
processes on the half line $[0,\infty )$ by using the Stieltjes transforms.
\end{remark}

\subsection{Off-diagonal estimates}

In order to state the estimates for $p\left( t,x,y\right) $ for arbitrary $%
x,y\in M$, we need some notation. For any $x\in M$ set%
\begin{equation*}
\left\vert x\right\vert :=d\left( x,K\right) +e.
\end{equation*}%
For all $x\in M$ and for all $t>2$, define the following functions: 
\begin{equation}
D(x,t):=\left\{ 
\begin{array}{ll}
1, & \text{if }\left\vert x\right\vert >\sqrt{t}\ \text{and }x\in E_{i}, \\ 
\frac{\left\vert x\right\vert ^{2}V_{i}(\sqrt{t})}{tV_{i}(\left\vert
x\right\vert )}, & \text{if }\left\vert x\right\vert \leq \sqrt{t}\text{ and 
}x\in E_{i}, \\ 
0, & \text{if }x\in K,%
\end{array}%
\right.  \label{function D}
\end{equation}%
\begin{equation}
U\left( x,t\right) :=\left\{ 
\begin{array}{ll}
\frac{1}{\log \left\vert x\right\vert }, & \text{if\ }\left\vert
x\right\vert >\sqrt{t} \\ 
\frac{1}{\log \sqrt{t}}\log \frac{e\sqrt{t}}{|x|}, & \text{if }\left\vert
x\right\vert \leq \sqrt{t},%
\end{array}%
\right.  \label{function U}
\end{equation}%
\begin{equation}
W(x,t):=\left\{ 
\begin{array}{ll}
1, & \text{if }\left\vert x\right\vert >\sqrt{t} \\ 
\frac{\log \left\vert x\right\vert }{\log \sqrt{t}}, & \text{if }\left\vert
x\right\vert \leq \sqrt{t}.%
\end{array}%
\right.  \label{function W}
\end{equation}

It is clear that $U\left( x,t\right) \leq 1$, $U\left( x,t\right) \nearrow 1$
as $t\rightarrow $ $\infty ,$ and $W\left( x,t\right) \leq 1$ and $W\left(
x,t\right) \searrow 0\ $as $t\rightarrow \infty .$ It is also useful to
observe that%
\begin{equation}
1\leq U\left( x,t\right) +W\left( x,t\right) \leq 2.  \label{U+W}
\end{equation}%
If $V_{i}\left( r\right) $ is either critical or subcritical, then it is
possible to show that $D\left( x,t\right) $ is bounded.

The next three theorems constitute our second main result. It is obtained by
combining Theorem \ref{main theorem} with several results from \cite{G-SC
Dirichlet}, \cite{G-SC hitting} and \cite{G-SC ends}.

In the first theorem we consider the case when $x$ and $y$ lie at different
ends.

\begin{theorem}
\label{T1}In the setting of Theorem {\ref{main theorem}}, the following
estimates are true for all $x\in E_{i}$, $y\in E_{j}$ with $i\neq j$ and $%
t>t_{0}$, where $t_{0}$ is large enough.

\begin{enumerate}
\item[$\left( i\right) $] If all the manifolds $M_{l}$, $l=1,...,k$, are
subcritical then%
\begin{equation}
p(t,x,y)\asymp \frac{C}{V_{\max }(\sqrt{t})}e^{-b\frac{d^{2}\left(
x,y\right) }{t}}.  \label{T1i}
\end{equation}

\item[$\left( ii\right) $] Suppose that at least one of the manifolds $M_{l}$%
, $l=1,...,k$, is critical.

$\left( ii\right) _{1}$ If both of $M_{i}$ and $M_{j}$ are subcritical, then 
\begin{equation}
p(t,x,y)\asymp \frac{C}{t}\left( 1+\left( D(x,t)+D(y,t)\right) \log t\right)
e^{-b\frac{d^{2}\left( x,y\right) }{t}}.  \label{T1ii1}
\end{equation}

$\left( ii\right) _{2}$ If both of $M_{i}$ and $M_{j}$ are critical, then 
\begin{equation}
p(t,x,y)\asymp \frac{C}{t}\left(
U(x,t)U(y,t)+W(x,t)U(y,t)+U(x,t)W(y,t)\right) e^{-b\frac{d^{2}\left(
x,y\right) }{t}}.  \label{T1ii2}
\end{equation}%
$\left( ii\right) _{3}$ If $M_{i}$ is subcritical and $M_{j}$ is critical,
then 
\begin{equation}
p(t,x,y)\asymp \frac{C}{t}\left( 1+D(x,t)U(y,t)\log t\right) e^{-b\frac{%
d^{2}\left( x,y\right) }{t}}.  \label{T1ii3}
\end{equation}
\end{enumerate}
\end{theorem}

The next two theorems cover the case when $x,y$ lie at the same end.

\begin{theorem}
\label{T2}In the setting of Theorem \ref{main theorem}, assume that $x,y\in
E_{i}$ and $t>t_{0}$.

\begin{enumerate}
\item[$\left( a\right) $] If $\sqrt{t}\leq \min \left( \left\vert
x\right\vert ,\left\vert y\right\vert \right) $ then%
\begin{equation}
p(t,x,y)\asymp \frac{C}{V_{i}(x,\sqrt{t})}e^{-b\frac{d^{2}\left( x,y\right) 
}{t}}.  \label{Vi}
\end{equation}

\item[$\left( b\right) $] Moreover, if $V_{i}\left( r\right) \approx V_{\max
}\left( r\right) $ for all large $r$, then (\ref{Vi}) holds for all $t>t_{0}$%
. In particular, this is the case when $M_{i}$ is critical.
\end{enumerate}
\end{theorem}

Estimate (\ref{Vi}) means that, for a restricted time, Brownian motion on
each end does not see the other ends, which is natural to expect. Note that
the same phenomenon holds also in the case when all $M_{i}$ are
non-parabolic.

The second claim of Theorem \ref{T2} means that, on the maximal end,
Brownian motion does not see the other ends for all times. It is interesting
to observe that in the case when all $M_{i}$ are non-parabolic, a similar
statement holds for the minimal end.

\begin{theorem}
\label{T3}In the setting of Theorem \ref{main theorem}, assume that $M_{i}$
is subcritical, $x,y\in E_{i}$ and $t>t_{0}$. If $\sqrt{t}\geq \min \left(
\left\vert x\right\vert ,\left\vert y\right\vert \right) $ then the
following is true.

\begin{enumerate}
\item[$\left( i\right) $] If all the manifolds $M_{l}$, $l=1,...,k$, are
subcritical, then%
\begin{equation}
p(t,x,y)\asymp C\left( \frac{D(x,t)D(y,t)}{V_{i}(\sqrt{t})}+\frac{1}{V_{\max
}(\sqrt{t})}\right) e^{-b\frac{d^{2}\left( x,y\right) }{t}}.  \label{T3i}
\end{equation}

\item[$\left( ii\right) $] If at least one of the manifolds $M_{l}$, $%
l=1,...,k$, is critical then 
\begin{equation}
p(t,x,y)\asymp C\left( \frac{D(x,t)D(y,t)}{V_{i}(\sqrt{t})}+\frac{1}{t}%
\left( 1+\left( D(x,t)+D(y,t)\right) \log t\right) \right) e^{-b\frac{%
d^{2}\left( x,y\right) }{t}}.  \label{T3ii}
\end{equation}
\end{enumerate}
\end{theorem}

\begin{remark}
\RM All the estimates of Theorems \ref{T1}-\ref{T3} can be extended to all $%
x,y\in M$ including also a possibility $x\in K$ or $y\in K$. This follows
from the local Harnack inequality for the heat kernel $p(t,x,y)$ and from a
careful analysis of the estimates. The latter shows that in all cases when $%
\left\vert x\right\vert $ (or $\left\vert y\right\vert $) remains bounded,
the terms containing $D\left( x,t\right) $ are dominated by others and,
hence, can be eliminated, which is equivalent to setting $D\left( x,t\right)
=0$ as in (\ref{function D}). A graphical summary of the estimates of
Theorems \ref{T1}-\ref{T3} can be found at the following location:

\href{https://www.math.uni-bielefeld.de/~grigor/tables.pdf}{%
https://www.math.uni-bielefeld.de/\symbol{126}grigor/tables.pdf}

%\begin{equation*}
%\mathrm{https}\text{\textrm{:}}\mathrm{//www.math.uni}\text{\textrm{-}}%
%\mathrm{bielefeld.de/\symbol{126}grigor/pubs.htm}
%\end{equation*}
\end{remark}

\begin{remark}
\RM By \cite[Lemma 5.9]{G-SC ends}, for all $x,y\in M$ and $0<t\leq t_{0}$,
the heat kernel on $M$ satisfies the Li-Yau estimate $($\ref{LY type}$)$
with constants depending on $t_{0}$. For this result it suffices to assume
that each end $M_{i}$ satisfies the Li-Yau estimate. Hence, in Theorems \ref%
{T1}-\ref{T3} we do not worry about the estimates for $t\leq t_{0}$.
\end{remark}

%%%%%%%%%%%%%%%%%%%%% Corollary : Simple case  %%%%%%%%%%%%%%%%%%%%%%%%%%%%%
If $V_{i}(r)$ is a power function for each $i=1,\ldots k$, then we can
simplify the heat kernel estimates of Theorems \ref{T1}-\ref{T3} as follows.
In the next statement $x,y$ lie at different ends.

\begin{corollary}
\label{power} Suppose that $V_{i}(r)\approx r^{\alpha _{i}}$ for all $%
i=1,\ldots ,k$ and $r\geq 1$, where $0<\alpha _{i}\leq 2$.

\begin{enumerate}
\item[$\left( i\right) $] Assume that $0<\alpha _{i}<2$ for all $i=1,...,k$
and set 
\begin{equation*}
\alpha =\max_{1\leq i\leq k}\alpha _{i}~.
\end{equation*}%
Then, for all $x,y$ lying at different ends and for all $t>2$, we have%
\begin{equation*}
p\left( t,x,y\right) \asymp \frac{C}{t^{\alpha /2}}e^{-b\frac{d^{2}\left(
x,y\right) }{t}}.
\end{equation*}

\item[$\left( ii\right) $] Assume that $\alpha _{l}=2$ for some $1\leq l\leq
k$. Then the following estimates hold for $i\neq j$, $x\in E_{i}$, $y\in
E_{j}$, $t>2$.
\end{enumerate}

$\left( ii\right) _{1}$ Let $\alpha _{i}<2$ and $\alpha _{j}<2$. If $\min
(\left\vert x\right\vert ,\left\vert y\right\vert )\geq \sqrt{t}$ then%
\begin{equation*}
p\left( t,x,y\right) \asymp \frac{C\log t}{t}e^{-b\frac{d^{2}\left(
x,y\right) }{t}},
\end{equation*}%
and if $\min (\left\vert x\right\vert ,\left\vert y\right\vert )\leq \sqrt{t}
$ then%
\begin{equation*}
p\left( t,x,y\right) \asymp \frac{C}{t}\left( 1+\log t~\left[ \left( \frac{%
\left\vert x\right\vert }{\sqrt{t}}\right) ^{2-\alpha _{i}}+\left( \frac{%
\left\vert y\right\vert }{\sqrt{t}}\right) ^{2-\alpha _{j}}\right] \right) .
\end{equation*}

$\left( ii\right) _{2}$ If $\alpha _{i}=\alpha _{j}=2$ then%
\begin{equation*}
p\left( t,x,y\right) \asymp \frac{C}{t}\left( U\left( x,t\right) U\left(
y,t\right) +U\left( x,t\right) \frac{\log |y|}{\log |y|+\log t}+U\left(
y,t\right) \frac{\log |x|}{\log |x|+\log t}\right) e^{-b\frac{d^{2}\left(
x,y\right) }{t}}.
\end{equation*}%
Consequently, if $\left\vert x\right\vert ,\left\vert y\right\vert \geq 
\sqrt{t}$ then%
\begin{equation}
p\left( t,x,y\right) \asymp \frac{C}{t}\left( \frac{1}{\log \left\vert
x\right\vert }+\frac{1}{\log \left\vert y\right\vert }\right) e^{-b\frac{%
d^{2}\left( x,y\right) }{t}},  \label{power (ii)_2}
\end{equation}%
if $\left\vert x\right\vert ,\left\vert y\right\vert \leq \sqrt{t}$ then%
\begin{equation}
p\left( t,x,y\right) \asymp \frac{C}{t\log ^{2}t}\left( \log t+\log ^{2}%
\sqrt{t}-\log |x|\log |y|\right) ,  \label{power (ii)_3}
\end{equation}%
and if $|x|\geq \sqrt{t}\geq |y|$ then%
\begin{equation}
p(t,x,y)\asymp \frac{C}{t\log t}\log \frac{e\sqrt{t}}{|y|}e^{-b\frac{%
d^{2}(x,y)}{t}}.  \label{x>y}
\end{equation}%
Similarly, if $|y|\geq \sqrt{t}\geq |x|$ then%
\begin{equation}
p(t,x,y)\asymp \frac{C}{t\log t}\log \frac{e\sqrt{t}}{|x|}e^{-b\frac{%
d^{2}(x,y)}{t}}.  \label{y>x}
\end{equation}

$\left( ii\right) _{3}$ If $\alpha _{i}<2$ and $\alpha _{j}=2$ then%
\begin{equation}
p\left( t,x,y\right) \asymp \frac{C}{t}\left( 1+\left( \frac{\left\vert
x\right\vert }{\left\vert x\right\vert +\sqrt{t}}\right) ^{2-\alpha
_{i}}U\left( y,t\right) \log t~\right) e^{-b\frac{d^{2}\left( x,y\right) }{t}%
}.  \label{Corii3}
\end{equation}%
Consequently, if $\left\vert y\right\vert \geq \sqrt{t}$ then%
\begin{equation*}
p\left( t,x,y\right) \asymp \frac{C}{t}e^{-b\frac{d^{2}\left( x,y\right) }{t}%
},
\end{equation*}%
if $\left\vert x\right\vert ,\left\vert y\right\vert \leq \sqrt{t}$ then%
\begin{equation*}
p\left( t,x,y\right) \asymp \frac{C}{t}\left( 1+\left( \frac{\left\vert
x\right\vert }{\sqrt{t}}\right) ^{2-\alpha _{i}}\log \frac{e\sqrt{t}}{%
\left\vert y\right\vert }\right) ,
\end{equation*}%
and if $|x|\geq \sqrt{t}\geq |y|$ then%
\begin{equation*}
p(t,x,y)\asymp \frac{C}{t}\log \frac{e\sqrt{t}}{|y|}e^{-b\frac{d^{2}(x,y)}{t}%
}.
\end{equation*}
\end{corollary}

\begin{proof}
All the estimates of Corollary \ref{power} follow immediately from those of
Theorem \ref{T1} and the definitions of functions $D$ and $W$. In the case $%
\left( ii\right) _{2}$, in the range $\left\vert x\right\vert ,\left\vert
y\right\vert \leq \sqrt{t}$, Theorem \ref{T1} gives the estimate 
\begin{equation*}
p\left( t,x,y\right) \asymp \frac{C}{t\log ^{2}\sqrt{t}}\left( \log \frac{e%
\sqrt{t}}{\left\vert x\right\vert }\log \frac{e\sqrt{t}}{\left\vert
y\right\vert }+\log \left\vert y\right\vert \log \frac{e\sqrt{t}}{\left\vert
x\right\vert }+\log \left\vert x\right\vert \log \frac{e\sqrt{t}}{\left\vert
y\right\vert }\right) .
\end{equation*}%
Since the sum in the brackets is equal to  
\begin{equation*}
\left( \log \left\vert x\right\vert +\log \frac{e\sqrt{t}}{\left\vert
x\right\vert }\right) \left( \log \left\vert y\right\vert +\log \frac{e\sqrt{%
t}}{\left\vert y\right\vert }\right) -\log \left\vert x\right\vert \log
\left\vert y\right\vert =\left( 1+\log \sqrt{t}\right) ^{2}-\log \left\vert
x\right\vert \log \left\vert y\right\vert ,
\end{equation*}%
we obtain (\ref{power (ii)_3}).
\end{proof}

%%%%%%%%%%%%%%%% Corollary:   Estimates in some regimes  %%%%%%%%%%%%%%%%%%%%
Let us state some consequences of Theorems \ref{T1}-\ref{T3} in the general
setting, but under some specific restrictions of the variables $x,y,t$.

\begin{corollary}
\label{Corcases}Under the hypotheses of Theorems \emph{\ref{T1}-\ref{T3}},
we have the following estimates.

\begin{enumerate}
\item[$\left( a\right) $] (Long time regime) For fixed $x,y\in M$ and $%
t\rightarrow \infty $, 
\begin{equation}
p\left( t,x,y\right) \approx \frac{1}{V_{\max }\left( \sqrt{t}\right) }.
\label{ptmax}
\end{equation}

\item[$\left( b\right) $] (Medium time regime) Let $x\in E_{i}$ and $y\in
E_{j}$ with $i\neq j$. If $\left\vert x\right\vert \approx \left\vert
y\right\vert \approx \sqrt{t}$ then in the cases $\left( i\right) $ and $%
\left( ii\right) _{3}$ we have (\ref{ptmax}), in the case $\left( ii\right)
_{1}$ we have%
\begin{equation}
p\left( t,x,y\right) \approx \frac{\log t}{t},  \label{logt/t}
\end{equation}%
and in the case $\left( ii\right) _{2}$%
\begin{equation}
p\left( t,x,y\right) \approx \frac{1}{t\log t}.  \label{tlogt}
\end{equation}
\end{enumerate}
\end{corollary}

\begin{proof}
$\left( a\right) $ The estimate (\ref{ptmax}) follows easily from Theorem %
\ref{main theorem} by using a local Harnack inequality. However, we show
here how it follows from Theorems \ref{T1}, \ref{T3}. Observe that, for a
fixed $x\in E_{i}$ and large $t$ we have 
\begin{equation}
D\left( x,t\right) \approx \frac{V_{i}\left( \sqrt{t}\right) }{t},\text{\ \ }%
U\left( x,t\right) \approx 1,\ \ W\left( x,t\right) \approx \frac{1}{\log t}.
\label{DUW}
\end{equation}%
Assume that $x\in E_{i}$, $y\in E_{j}$ and consider the cases $\left(
i\right) ,\left( ii\right) _{1},\left( ii\right) _{2}$ and $\left( ii\right)
_{3}$ as in Theorem \ref{T1}.

Case $\left( i\right) $. Using (\ref{T1i}), (\ref{T3i}), (\ref{DUW}) and $%
V_{i}\left( x,\sqrt{t}\right) \approx V_{i}\left( \sqrt{t}\right) \ $as $%
t\rightarrow \infty $ we obtain 
\begin{equation*}
p\left( t,x,y\right) \approx \frac{V_{i}\left( \sqrt{t}\right) }{t^{2}}%
\delta _{ij}+\frac{1}{V_{\max }\left( \sqrt{t}\right) }\approx \frac{1}{%
V_{\max }\left( \sqrt{t}\right) },
\end{equation*}
where we have also used that $V_{j}\left( r\right) V_{\max }\left( r\right)
=o\left( r^{4}\right) .$

Case $\left( ii\right) _{1}$. By (\ref{T1ii1}), (\ref{T3ii}) and (\ref{DUW})
we have%
\begin{eqnarray*}
p\left( t,x,y\right) &\approx &\frac{V_{i}\left( \sqrt{t}\right) }{t^{2}}%
\delta _{ij}+\frac{1}{t}\left\{ 1+\left( \frac{V_{i}(\sqrt{t})}{t}+\frac{%
V_{j}(\sqrt{t})}{t}\right) \log t\right\} \\
&\approx &\frac{1}{t}\approx \frac{1}{V_{\max }\left( \sqrt{t}\right) },
\end{eqnarray*}%
because of $V_{\max }\left( r\right) \approx r^{2}$ and (\ref{or2logr})$.$

Case $\left( ii\right) _{2}.$ If $i\neq j$ then by (\ref{T1ii2}) and (\ref%
{DUW})%
\begin{equation*}
p(t,x,y)\approx \frac{1}{t}\left( 1+\frac{1}{\log t}\right) \approx \frac{1}{%
t}\approx \frac{1}{V_{\max }\left( \sqrt{t}\right) }.
\end{equation*}%
If $i=j$ then (\ref{ptmax}) follows trivially from (\ref{Vi}).

Case $\left( ii\right) _{3}.$ In this case necessarily $i\neq j$, and we
obtain by (\ref{T1ii3}) 
\begin{equation*}
p\left( t,x,y\right) \approx \frac{1}{t}\left\{ 1+\frac{V_{i}(\sqrt{t})}{t}%
\log t\right\} \approx \frac{1}{t}\approx \frac{1}{V_{\max }\left( \sqrt{t}%
\right) }.
\end{equation*}

$\left( b\right) $ In the case $\left\vert x\right\vert \approx \left\vert
y\right\vert \approx \sqrt{t}$ we have $d^{2}\left( x,y\right) \approx t$ and%
\begin{equation*}
D\left( x,t\right) \approx 1,\ \ \ \ U\left( x,t\right) \approx \frac{1}{%
\log t},\ \ \ W\left( x,t\right) \approx 1.
\end{equation*}%
Then the required estimates follow directly from those stated in Theorem \ref%
{T1}.
\end{proof}

Let us observe the following. In the medium time regime, that is, when
\thinspace $x$ and $y$ lie at different ends and $\left\vert x\right\vert
\approx \left\vert y\right\vert \approx \sqrt{t}$, we have by $\left(
b\right) $: in the cases $\left( i\right) $ and $\left( ii\right) _{3}$ 
\begin{equation*}
p\left( t,x,y\right) \approx \frac{1}{V_{\max }\left( \sqrt{t}\right) },
\end{equation*}%
that is, $p\left( t,x,y\right) $ behaves itself as in the long time regime,
whereas in the case $\left( ii\right) _{1}$%
\begin{equation*}
p\left( t,x,y\right) \approx \frac{\log t}{t}\gg \frac{1}{V_{\max }\left( 
\sqrt{t}\right) },
\end{equation*}%
and in the case $\left( ii\right) _{2}$%
\begin{equation*}
p\left( t,x,y\right) \approx \frac{1}{t\log t}\ll \frac{1}{V_{\max }\left( 
\sqrt{t}\right) }.
\end{equation*}

Hence, we observe in the case $\left( ii\right) _{2}$ the \textit{bottleneck
effect}: the heat kernel value $\frac{1}{t\log t}$ in the medium time regime
is significantly \textit{smaller} than that of long time regime $\frac{1}{t}$%
. For example, this case happens for $M=\mathbb{R}^{2}\#\mathbb{R}^{2}$ (see
Fig. \ref{rn+rn}). A similar bottleneck effect was observed in \cite{G-SC
ends} for $M=\mathbb{R}^{n}\#\mathbb{R}^{n}$ with $n\geq 3$: the heat kernel
of $M$ in the long time regime is comparable to $\frac{1}{t^{n/2}}$ whereas
in the medium time regime -- to $\frac{1}{t^{n-1}}$. In the case $n=2$ the
bottleneck effect is quantitatively weaker as the distinction between the
two regimes is determined by $\log t$ in contrast to the power of $t$ in the
case $n\geq 3.$

On the contrary, in the case $\left( ii\right) _{1}$ we observe an
interesting \textit{anti-bottleneck effect}: the heat kernel value $\frac{%
\log t}{t}$ in the medium time regime is significantly \textit{larger} than
that of the long time regime $\frac{1}{t}$. This effect occurs only when
there are at least three ends, one of them being critical and two --
subcritical. For example, this is the case for $M=\mathcal{R}^{1}\#\mathcal{R%
}^{1}\#\mathcal{R}^{2}$ (see Fig. \ref{r2+r1}).

\FRAME{ftbphFU}{3.0225in}{1.6016in}{0pt}{\Qcb{Connected sum $\mathcal{R}%
^{1}\#\mathcal{R}^{1}\#\mathcal{R}^{2}$}}{\Qlb{r2+r1}}{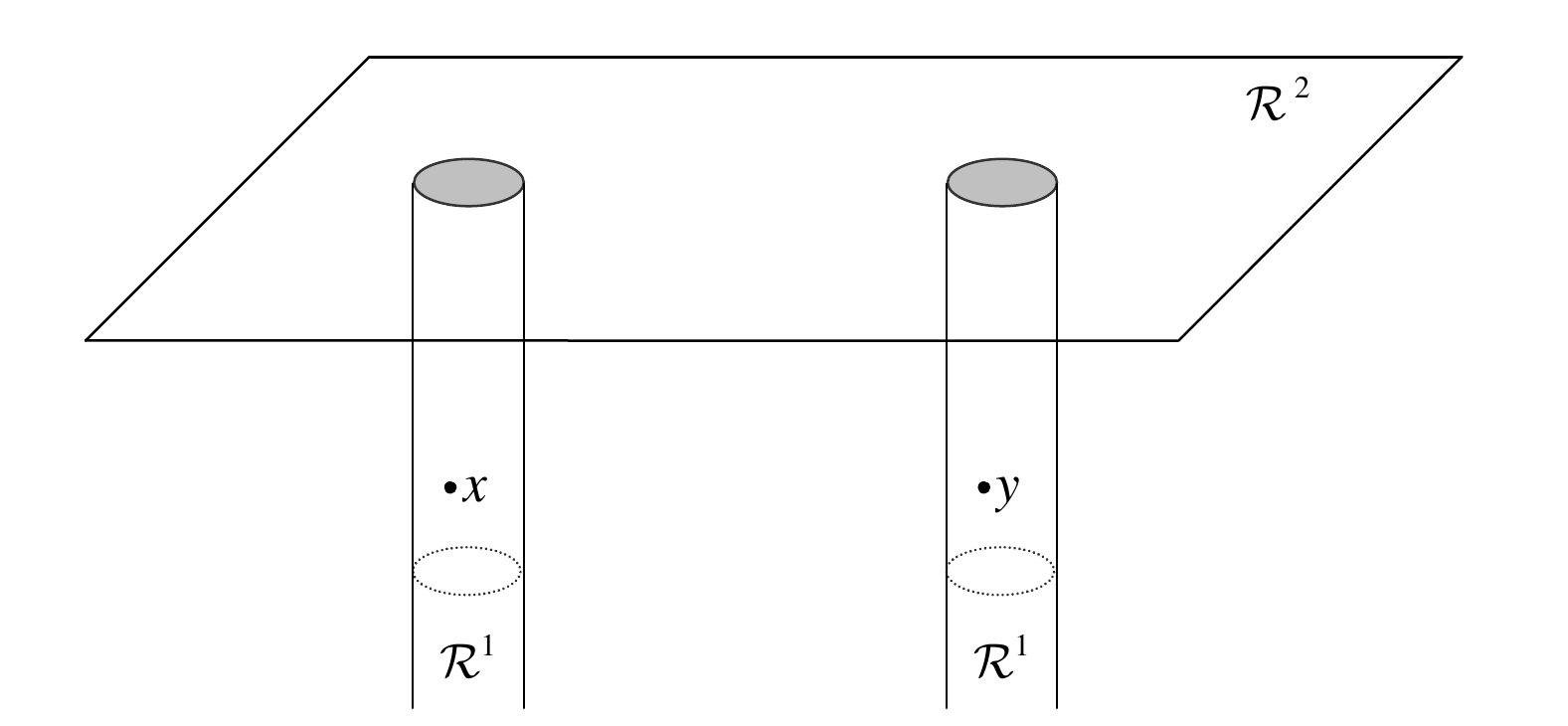}{\special%
{language "Scientific Word";type "GRAPHIC";maintain-aspect-ratio
TRUE;display "USEDEF";valid_file "F";width 3.0225in;height 1.6016in;depth
0pt;original-width 6.333in;original-height 3.3356in;cropleft "0";croptop
"1";cropright "1";cropbottom "0";filename 'r2+r1.eps';file-properties
"XNPEU";}}

\subsection{Examples}

\label{examples}

We present here heat kernel bounds on some specific examples using Theorems %
\ref{T1}-\ref{T3} and Corollary \ref{power}.

\begin{example}[Heat kernel on $\mathcal{R}^{\protect\alpha _{1}}\#\mathcal{R%
}^{\protect\alpha _{2}}$]
\RM Let us write down the heat kernel bounds on the connected sum 
\begin{equation*}
M=M_{1}\#M_{2}=\mathcal{R}^{\alpha _{1}}\#\mathcal{R}^{\alpha _{2}},
\end{equation*}%
where $1\leq \alpha _{1}\leq \alpha _{2}<2$. In this case both $M_{1}$ and $%
M_{2}$ are subcritical so that Theorem \ref{T1}$\left( i\right) $, Theorem %
\ref{T2} and Theorem \ref{T3}$\left( i\right) $ apply. Observe that 
\begin{equation}
D(x,t)=\left\{ 
\begin{array}{ll}
1, & \mbox{if }\left\vert x\right\vert >\sqrt{t}, \\ 
\left( \frac{\left\vert x\right\vert }{\sqrt{t}}\right) ^{2-\alpha _{i}}, & %
\mbox{if }\left\vert x\right\vert \leq \sqrt{t},%
\end{array}%
\right.  \label{Di}
\end{equation}%
and%
\begin{equation*}
V_{\max }\left( r\right) \approx r^{\alpha _{2}},\ \ \ r>1.
\end{equation*}%
In the case $x\in E_{1}$ and $y\in E_{2}$, we obtain by (\ref{T1i}) or by
Corollary \ref{power}$\left( i\right) $, 
\begin{equation*}
p(t,x,y)\asymp \frac{C}{t^{\alpha _{2}/2}}e^{-b\frac{d^{2}\left( x,y\right) 
}{t}}.
\end{equation*}%
\newline

Assume now that $x,y\in E_{1}$. If $\left\vert x\right\vert ,\left\vert
y\right\vert >\sqrt{t}$, then by (\ref{Vi}) we have%
\begin{equation*}
p\left( t,x,y\right) \asymp \frac{C}{V_{1}(x,\sqrt{t})}e^{-b\frac{%
d^{2}\left( x,y\right) }{t}}.
\end{equation*}
If $\left\vert x\right\vert ,\left\vert y\right\vert \leq \sqrt{t}$ then by (%
\ref{T3i}) and (\ref{Di}) we obtain%
\begin{equation}
p(t,x,y)\approx \frac{1}{t^{\alpha _{1}/2}}\left( \frac{\left\vert
x\right\vert \left\vert y\right\vert }{t}\right) ^{2-\alpha _{1}}+\frac{1}{%
t^{\alpha _{2}/2}}.  \label{a1a2}
\end{equation}%
In particular, in the long time regime $t\rightarrow \infty $ we obtain%
\begin{equation*}
p\left( t,x,y\right) \approx \frac{1}{t^{\alpha _{2}/2}},
\end{equation*}%
which, of course, matches (\ref{ptmax}). Assume now that $\left\vert
x\right\vert >\sqrt{t}\geq \left\vert y\right\vert $. Substituting (\ref{Di}%
) into (\ref{T3i}), we obtain%
\begin{equation*}
p\left( t,x,y\right) \asymp C\left( \frac{1}{t^{\alpha_1/2}}\left( \frac{%
\left\vert y\right\vert }{\sqrt{t}}\right) ^{2-\alpha _{1}}+\frac{1}{%
t^{\alpha _{2}/2}}\right) e^{-b\frac{d^{2}(x,y)}{t}}.
\end{equation*}%
A similar estimate holds in the case $|y|>\sqrt{t}\geq |x|$.

Finally, if $x,y\in E_{2}$ then we have by Theorem \ref{T2} that for all $%
t>1 $%
\begin{equation*}
p(t,x,y)\asymp \frac{C}{V_{2}(x,\sqrt{t})}e^{-b\frac{d^{2}\left( x,y\right) 
}{t}}.
\end{equation*}
\end{example}

\begin{example}[Heat kernel on $\mathcal{R}^{1}\#\mathcal{R}^{2}$]
\RM Consider $M=M_{1}\#M_{2}=\mathcal{R}^{1}\#\mathcal{R}^{2}$ (see Fig. \ref%
{r1+r2}). Suppose that $x\in E_{1}$, $y\in E_{2}$. Then by Theorem \ref{T1}$%
\left( ii\right) _{3}$ or by the estimate (\ref{Corii3}) of Corollary \ref%
{power}%
\begin{equation*}
p(t,x,y)\asymp \frac{C}{t}\left( 1+\frac{|x|}{|x|+\sqrt{t}}U(y,t)\log
t\right) e^{-b\frac{d^{2}\left( x,y\right) }{t}}.
\end{equation*}%
Using (\ref{function U}) we obtain: if $|y|>\sqrt{t}$, then 
\begin{equation*}
p(t,x,y)\asymp \frac{C}{t}e^{-b\frac{d^{2}(x,y)}{t}};
\end{equation*}%
if $|x|,|y|\leq \sqrt{t}$, then%
\begin{equation*}
p(t,x,y)\asymp \frac{C}{t}\left( 1+\frac{|x|}{\sqrt{t}}\log \frac{e\sqrt{t}}{%
|y|}\right) ,
\end{equation*}%
and if $|x|>\sqrt{t}\geq |y|$, then%
\begin{equation*}
p(t,x,y)\asymp \frac{C}{t}\log \frac{e\sqrt{t}}{|y|}e^{-b\frac{d^{2}(x,y)}{t}%
}.
\end{equation*}%
\newline

Assume that $x,y\in E_{1}$. If $\min (\left\vert x\right\vert ,\left\vert
y\right\vert )\leq \sqrt{t}$, then we obtain by (\ref{T3ii}) and (\ref{Di})%
\begin{equation*}
p(t,x,y)\approx \frac{1}{t}\left( 1+\frac{\left\vert x\right\vert \left\vert
y\right\vert }{\sqrt{t}}+\frac{\left\vert x\right\vert +\left\vert
y\right\vert }{\sqrt{t}}\log t\right) e^{-b\frac{d^{2}(x,y)}{t}}.
\end{equation*}%
In particular, if $|x|>\sqrt{t}\geq |y|$, we obtain 
\begin{equation*}
p(t,x,y)\asymp \frac{C}{t}\left( |y|+\log t\right) e^{-b\frac{d^{2}(x,y)}{t}%
}.
\end{equation*}%
Similar estimate follows when $|y|>\sqrt{t}\geq |x|$. If $\min (|x|,|y|)>%
\sqrt{t}$, we obtain by Theorem \ref{T2} 
\begin{equation*}
p(t,x,y)\asymp \frac{C}{\sqrt{t}}e^{-b\frac{d^{2}(x,y)}{t}}.
\end{equation*}

In the case $x,y\in E_{2}$, we obtain by Theorem \ref{T2} 
\begin{equation}
p(t,x,y)\asymp \frac{C}{t}e^{-b\frac{d^{2}(x,y)}{t}} .  \label{ptcb}
\end{equation}
\end{example}

\begin{example}[Heat kernel on $\mathbb{R}^{2}\#\mathbb{R}^{2}$]
Suppose that $x\in E_{1}$ and $y\in E_{2}$. If $\left\vert x\right\vert
,\left\vert y\right\vert \leq \sqrt{t}$, then by (\ref{T1ii2}), or by (\ref%
{power (ii)_3}) 
\begin{equation*}
p(t,x,y)\approx \frac{1}{t\log ^{2}t}\left( \log t+\log ^{2}\sqrt{t}-\log
|x|\log |y|\right) .
\end{equation*}%
In particular, in the long time regime $\left\vert x\right\vert \approx
\left\vert y\right\vert \approx 1$ we obtain%
\begin{equation*}
p\left( t,x,y\right) \approx \frac{1}{t},
\end{equation*}%
and in the medium time regime $\left\vert x\right\vert \approx \left\vert
y\right\vert \approx \sqrt{t}$ we have%
\begin{equation*}
p(t,x,y)\approx \frac{1}{t\log t},
\end{equation*}%
which means a mild bottleneck-effect on $\mathbb{R}^{2}\#\mathbb{R}^{2}$. 
\newline

If $\left\vert x\right\vert ,\left\vert y\right\vert \geq \sqrt{t}$ then the
heat kernel on $\mathbb{R}^{2}\#\mathbb{R}^{2}$ satisfies (\ref{power (ii)_2}%
), that is, 
\begin{equation*}
p\left( t,x,y\right) \asymp \frac{C}{t}\left( \frac{1}{\log \left\vert
x\right\vert }+\frac{1}{\log \left\vert y\right\vert }\right) e^{-b\frac{%
d^{2}\left( x,y\right) }{t}}.
\end{equation*}%
The cases $\left\vert x\right\vert >\sqrt{t}\geq \left\vert y\right\vert $
and $\left\vert y\right\vert >\sqrt{t}\geq \left\vert x\right\vert $ are
covered by (\ref{x>y}) and (\ref{y>x}), respectively.

If $x,y\in E_{1}$ or $x,y\in E_{2}$ then $p\left( t,x,y\right) $ satisfies (%
\ref{ptcb}) by Theorem \ref{T2}.
\end{example}

\begin{example}[Heat kernel on $\mathcal{R}^{1}\#\mathcal{R}^{1}\#\mathcal{R}%
^{2}$]
\RM Let $M=M_{1}\#M_{2}\#M_{3}=\mathcal{R}^{1}\#\mathcal{R}^{1}\#\mathcal{R}%
^{2}$ (see Fig. \ref{r2+r1}). If $x$ and $y$ are at the same end, or $x\in 
\mathcal{R}^{1}$ and $y\in \mathcal{R}^{2}$, then the heat kernel $p(t,x,y)$
satisfies same estimates as in the above case $\mathcal{R}^{1}\#\mathcal{R}%
^{2}$.

Assume now that $x\in E_{1}$ and $y\in E_{2}$. Then by Corollary \ref{power}$%
\left( ii\right) _{1}$ we obtain the following estimates: if $\min
(\left\vert x\right\vert ,\left\vert y\right\vert )\leq \sqrt{t}$ then 
\begin{equation*}
p(t,x,y)\approx \frac{1}{t}\left( 1+\frac{\log t}{\sqrt{t}}(\left\vert
x\right\vert +\left\vert y\right\vert )\right) ,
\end{equation*}%
and if $\min (|x|,|y|)>\sqrt{t}$, then 
\begin{equation*}
p(t,x,y)\asymp \frac{\log t}{t}e^{-b\frac{d^{2}(x,y)}{t}}.
\end{equation*}%
In particular, if $\left\vert x\right\vert \approx \left\vert y\right\vert
\approx \sqrt{t}$, then 
\begin{equation*}
p(t,x,y)\approx \frac{\log t}{t}.
\end{equation*}
\end{example}

\section{Some auxiliary estimates}

\setcounter{equation}{0}\label{SecGeneral}In this section we prove some
auxiliary results to be used in the proof of Theorem \ref{main theorem}.

Let $(M,\mu )$ be a geodesically complete non-compact weighted manifold (we
do not even assume parabolicity of $M$ unless it is clearly stated).  For
any open set $\Omega \subset M$, denote by $p_{\Omega }\left( t,x,y\right) $
the Dirichlet heat kernel in $\Omega $. Assume from now on that $\Omega $
has smooth boundary. Then $p_{\Omega }\left( t,x,y\right) =0$ whenever $x$
or $y$ belongs to $\partial \Omega $. Denote also by $P_{t}^{\Omega }$ the
associated heat semigroup. Denote as before by $(\{X_{t}\}_{t\geq 0},\{%
\mathbb{P}_{x}\}_{x\in M})$ Brownian motion on $M$. Let $\tau _{\Omega }$ be
the first exit time of $X_{t}$ from $\Omega $, that is,%
\begin{equation*}
\tau _{\Omega }=\inf \left\{ t>0:X_{t}\notin \Omega \right\} .
\end{equation*}%
Then, for any bounded continuous function $f$ on $M$,%
\begin{equation}
P_{t}^{\Omega }f\left( x\right) =\mathbb{E}_{x}\left( f\left( X_{t}\right)
1_{\left\{ \tau _{\Omega }>t\right\} }\right) .  \label{PE}
\end{equation}

\subsection{Integrated resolvent}

The resolvent operator $G_{\lambda }^{\Omega }$ is defined for any $\lambda
>0$ as an operator on non-negative measurable functions $f$ on $\Omega $ by%
\begin{equation*}
G_{\lambda }^{\Omega }f\left( x\right) =\int_{0}^{\infty }e^{-\lambda
t}P_{t}^{\Omega }f\,dt.
\end{equation*}%
Clearly, $G_{\lambda }^{\Omega }$ is a linear operator that preserves
non-negativity. Note that by definition $G_{\lambda }^{\Omega }f$ vanishes
in $\Omega ^{c}$. If $\Omega =M$ then we write $G_{\lambda }\equiv
G_{\lambda }^{M}$. Clearly, $G_{\lambda }^{\Omega }$ is an integral operator
whose kernel%
\begin{equation*}
g_{\lambda }^{\Omega }\left( x,y\right) =\int_{0}^{\infty }e^{-\lambda
t}p_{\Omega }\left( t,x,y\right) dt
\end{equation*}%
is called the \textit{resolvent kernel}. In general, $G_{\lambda }^{\Omega
}f $ may take value $+\infty $. However, if $f$ is bounded and continuous
then the function $u=G_{\lambda }^{\Omega }f$ is finite and, moreover, is
the minimal non-negative solution of the equation $\Delta u-\lambda u=-f$
(see \cite{G AMS}). It follows from (\ref{PE}) that%
\begin{equation}
G_{\lambda }^{\Omega }f\left( x\right) =\mathbb{E}_{x}\left( \int_{0}^{\tau
_{\Omega }}f\left( X_{t}\right) e^{-\lambda t}dt\right) .  \label{FC}
\end{equation}%
If in addition $\Omega $ is precompact then the function $u=G_{\lambda
}^{\Omega }f$ solves the Dirichlet problem%
\begin{equation*}
\left\{ 
\begin{array}{ll}
\Delta u-\lambda u=-f & \text{in }\Omega , \\ 
u=0 & \text{on }\partial \Omega .%
\end{array}%
\right.
\end{equation*}%
For the proof of Theorem \ref{main theorem} we need the notion of \textit{%
integrated resolvent.} Fix a compact set $K\subset M$ with non-empty
interior $\mathring{K}$ such that $K$ is the closure of $\mathring{K}$ and
the boundary $\partial K$ is smooth. Fix also once and for all a reference
point $o\in K$.

For any $\lambda >0$, define the function $\gamma _{\lambda }$ on $M$ by%
\begin{equation}
\gamma _{\lambda } (x):=G_{\lambda } 1_{K}\left( x\right)
=\int_{K}g_{\lambda } \left( x,z\right) d\mu \left( z\right)
=\int_{K}\int_{0}^{\infty }e^{-\lambda t}p \left( t,x,z\right) dz\,dt.
\label{Gla}
\end{equation}%
The function $\gamma _{\lambda }$ is called the integrated resolvent. Set
also%
\begin{equation}
\dot{\gamma}_{\lambda } =G_{\lambda } \gamma _{\lambda } .  \label{gadot}
\end{equation}%
It follows from the resolvent equation $G_{\alpha }-G_{\beta }=\left( \beta
-\alpha \right) G_{\alpha }G_{\beta }$ that 
\begin{equation}
\dot{\gamma}_{\lambda }=-\frac{\partial }{\partial \lambda }\gamma _{\lambda
}=\int_{K}\int_{0}^{\infty }te^{-\lambda t}p\left( t,x,z\right) dz\,dt.
\label{Gladot}
\end{equation}

\begin{lemma}
\label{Lemma of G}

\begin{itemize}
\item[$\left( i\right) $] If there exist positive constants $C,\lambda _{0}$
and a function $F:\mathbb{R}_{+}\rightarrow \mathbb{R}_{+}$ such that, for
some $x\in K$, 
\begin{equation}
\gamma _{\lambda }(x)\leq \frac{C}{\lambda F(\frac{1}{\sqrt{\lambda }})}%
\quad \text{for all }\lambda \in (0,\lambda _{0}],  \label{Gl<}
\end{equation}%
then there exist positive constants $C^{\prime },t_{0}$ such that 
\begin{equation}
p(t,o,o)\leq \frac{C^{\prime }}{F(\sqrt{t})}\quad \text{for all }t\geq t_{0}.
\label{on-diagonal from resolvent}
\end{equation}

\item[$\left( ii\right) $] If there exist positive constants $C,\lambda _{0}$
such that, for some $x\in K$,%
\begin{equation}
\dot{\gamma}_{\lambda }(x)\leq \frac{C}{\lambda }\quad \text{for all }%
\lambda \in (0,\lambda _{0}],  \label{Gdot<}
\end{equation}%
then there exist positive constants $C^{\prime },t_{0}$ such that 
\begin{equation*}
p(t,o,o)\leq \frac{C^{\prime }}{t}\quad \text{for all }t\geq t_{0}.
\end{equation*}
\end{itemize}
\end{lemma}

%%%%%%%%%%%%%%%%%%%%%%%%%%%%%%%%%%%%%%%%%%%%%%%%%

\begin{proof}
$\left( i\right) $ Set $\delta =(\mathrm{diam}K)^{2}$. By the local Harnack
inequality, there exit positive constants $c_{1},c_{2}$ such that, for all $%
x,z\in K$ and $s>2c_{2}\delta $, 
\begin{equation}
p(s,x,z)\geq c_{1}p\left( s-c_{2}\delta ,o,o\right) ,  \label{local Harnack}
\end{equation}%
which implies by (\ref{Gla}), for all $x\in K$,%
\begin{equation*}
\gamma _{\lambda }(x)\geq c_{1}\mathrm{vol}(K)\int_{2c_{2}\delta }^{\infty
}e^{-\lambda s}p(s-c_{2}\delta ,o,o)ds.
\end{equation*}%
Using the monotonicity of $p(s,o,o)$ with respect to $s$ (see \cite[%
Exercises 7.22]{G AMS}), we obtain, for $t\geq 4c_{2}\delta ,$%
\begin{align}
\gamma _{\lambda }\left( x\right) & \geq c_{1}\mathrm{vol}%
(K)\int_{t/2}^{t}e^{-\lambda s}p(s-c_{2}\delta ,o,o)ds  \notag \\
& \geq c_{1}\mathrm{vol}(K)\int_{t/2}^{t}e^{-\lambda s}p(t,o,o)ds\geq
cte^{-\lambda t}p(t,o,o).  \label{G2}
\end{align}%
Set $t_{0}:=\max \{4c_{2}\delta ,\lambda _{0}^{-1}\}$. For any $t\geq t_{0}$
and using (\ref{Gl<}) and (\ref{G2}) with $\lambda =t^{-1}$, we obtain%
\begin{equation*}
\frac{C}{\lambda F(\frac{1}{\sqrt{\lambda }})}\geq cte^{-1}p(t,o,o),
\end{equation*}%
which implies%
\begin{equation*}
p(t,o,o)\leq \frac{C}{F(\sqrt{t})}.
\end{equation*}

$\left( ii\right) $ Arguing as in $\left( i\right) $ and using (\ref{local
Harnack}) and (\ref{Gladot}), we obtain, for $t\geq 4c_{2}\delta $ and $x\in
K$,%
\begin{eqnarray}
\dot{\gamma}_{\lambda }(x) &=&\int_{0}^{\infty }\int_{K}se^{-\lambda
s}p(s,x,z)dsd\mu (z)  \notag \\
&\geq &c_{1}\mathrm{vol}(K)\int_{t/2}^{t}se^{-\lambda s}p(t,o,o)ds\geq
ct^{2}e^{-\lambda t}p(t,o,o).  \label{Gdot2}
\end{eqnarray}%
Assuming $t\geq t_{0}:=\max \{4c_{2}\delta ,\lambda _{0}^{-1}\}$ and using (%
\ref{Gdot<}) and (\ref{Gdot2}) with $\lambda =t^{-1}$, we obtain%
\begin{equation*}
\frac{C}{\lambda }\geq ct^{2}e^{-1}p(t,o,o),
\end{equation*}%
which implies%
\begin{equation*}
p(t,o,o)\leq \frac{C}{t}.
\end{equation*}
\end{proof}

\begin{remark}
\RM Lemma \ref{Lemma of G} will be used in the proof of Theorem \ref{main
theorem} in Section \ref{SecUpper} as follows. In the case when all the ends
are subcritical, we will prove the following upper bound for the integrated
resolvent:%
\begin{equation}
\sup_{\partial K}\gamma _{\lambda }\leq \frac{C}{\lambda V_{\max }(\frac{1}{%
\sqrt{\lambda }})},  \label{supga}
\end{equation}%
which then implies by Lemma \ref{Lemma of G}$\left( i\right) $ the desired
upper bound 
\begin{equation*}
p(t,o,o)\leq \frac{C}{V_{\max }(\sqrt{t})}.
\end{equation*}%
However, in the case when one of the ends is critical, we obtain instead of (%
\ref{supga}) a weaker inequality%
\begin{equation}
\sup_{\partial K}\gamma _{\lambda }\leq C\log \frac{1}{\lambda },
\label{logla}
\end{equation}%
which yields 
\begin{equation*}
p(t,o,o)\leq C\frac{\log t}{t}
\end{equation*}%
instead of the desired estimate 
\begin{equation}
p(t,o,o)\leq \frac{C}{t}.  \label{ptt}
\end{equation}%
In order to be able to prove the latter, we will use the second part of
Lemma \ref{Lemma of G}. Namely, we will prove that in the critical case%
\begin{equation}
\sup_{\partial K}\dot{\gamma}_{\lambda }\leq \frac{C}{\lambda },
\label{1/la}
\end{equation}%
which then will imply (\ref{ptt}) by Lemma \ref{Lemma of G}$\left( ii\right) 
$.

Note that the estimate (\ref{logla}) of $\gamma _{\lambda }$ is already
optimal as it is matched by the estimate (\ref{1/la}) of $\dot{\gamma}%
_{\lambda }=-\frac{\partial }{\partial \lambda }\gamma _{\lambda }$.
However, the function $\gamma _{\lambda }$ alone does not allow to recover
an optimal estimate of the heat kernel, while its $\lambda $-derivative $%
\dot{\gamma}_{\lambda }$ does.
\end{remark}

\subsection{Comparison principles}

Fix an open set $\Omega \subset M$ and $\lambda >0$. We say that a function $%
u$ is $\lambda $-harmonic in $\Omega $ if it satisfies in $\Omega $ the
equation $\Delta u-\lambda u=0.$ A function $u$ is called $\lambda $%
-superharmonic if $\Delta u-\lambda u\leq 0$. We will frequently use the
following minimum principle: if $\Omega $ is precompact, $u\in C\left( 
\overline{\Omega }\right) $ is $\lambda $-superharmonic in $\Omega $ and $%
u\geq 0$ on $\partial \Omega $ then $u\geq 0$ in $\Omega $. It implies the
comparison principle: if $u,v\in C\left( \overline{\Omega }\right) $, $u$ is 
$\lambda $-superharmonic in $\Omega $ and $v$ is $\lambda $-harmonic in $%
\Omega $ then%
\begin{equation}
u\geq v\ \text{on }\partial \Omega \ \ \Rightarrow \ u\geq v\ \text{in }%
\Omega .  \label{u>v}
\end{equation}

Let now $\Omega $ be an exterior domain, that is, $\Omega =F^{c}$ where $F$
is a compact subset of $M$. Let $v\in C\left( \overline{\Omega }\right) $ be
non-negative and $\lambda $-harmonic in $\Omega $.\ We say that $v$ is 
\textit{minimal} in $\Omega $ if there exists an exhaustion $\left\{
U_{k}\right\} $ of $M$ by precompact open sets $U_{k}\supset F$ and a
sequence $\left\{ v_{k}\right\} $ of functions $v_{k}\in C\left( \overline{%
U_{k}\setminus F}\right) $ that are non-negative and $\lambda $-harmonic in $%
U_{k}\setminus F$ and such that $v_{k}|_{\partial U_{k}}=0$ and $%
v_{k}\uparrow v$ in $\overline{\Omega }$. Then the following modification of
the comparison principle holds in $\Omega $: if $u,v\in C\left( \overline{%
\Omega }\right) $, $u$ is non-negative $\lambda $-superharmonic in $\Omega $
and $v$ is non-negative minimal $\lambda $-harmonic in $\Omega $ then (\ref%
{u>v}) is satisfied. Indeed, by the comparison principle in $U_{k}\setminus
F $ we obtain $u\geq v_{k}$ whence the claim follows.

We are left to mention that, for any non-negative bounded function $f$ with
compact support, the function $G_{\lambda }f$ is non-negative, minimal, $%
\lambda $-harmonic outside $\limfunc{supp}f$, since $G_{\lambda
}^{U_{k}}f\uparrow G_{\lambda }f$.

\subsection{Functions $\Phi _{\protect\lambda }^{\Omega }$ and $\Psi _{%
\protect\lambda }^{\Omega }$}

In any open set $\Omega \subset M$, consider a function%
\begin{equation}
\Phi _{\lambda }^{\Omega }:=\lambda G_{\lambda }^{\Omega }1=\int_{0}^{\infty
}\lambda e^{-\lambda t}P_{t}^{\Omega }1\,dt.  \label{Fidef}
\end{equation}%
Since $0\leq P_{t}^{\Omega }1\leq 1$, we see that 
\begin{equation}
0\leq \Phi _{\lambda }^{\Omega }\leq 1.  \label{Fi1}
\end{equation}%
It follows from (\ref{PE}) that 
\begin{equation}
\Phi _{\lambda }^{\Omega }(x)=\int_{0}^{\infty }\lambda e^{-\lambda t}%
\mathbb{P}_{x}(\tau _{\Omega }>t)dt.  \label{Fi}
\end{equation}

Let $A$ be a precompact open subset of $M$ with smooth boundary and let $%
K\subset A$. Set 
\begin{equation}
\gamma _{\lambda }^{A}(x):=G_{\lambda }^{A}1_{K}\left( x\right)
=\int_{K}g_{\lambda }^{A}\left( x,z\right) d\mu \left( z\right)
=\int_{K}\int_{0}^{\infty }e^{-\lambda t}p_{A}(t,x,z)dtd\mu (z) .  \label{GA}
\end{equation}

\begin{lemma}
\label{Lemga}$\left( a\right) $ The following inequality holds in $A$:%
\begin{equation}
\gamma _{\lambda }-\gamma _{\lambda }^{A}\leq (\sup_{\partial A}\gamma
_{\lambda })\left( 1-\Phi _{\lambda }^{A}\right) .  \label{gaA}
\end{equation}%
$\left( b\right) $ The following inequality holds in $K^{c}$:%
\begin{equation}
\gamma _{\lambda }\leq (\sup_{\partial K}\gamma _{\lambda })\left( 1-\Phi
_{\lambda }^{K^{c}}\right) .  \label{gaKc}
\end{equation}
\end{lemma}

\begin{proof}
$\left( a\right) $ By (\ref{Fidef}), the function $\Phi _{\lambda }^{A}$
satisfies%
\begin{equation*}
\left\{ 
\begin{array}{ll}
\Delta \Phi _{\lambda }^{A}-\lambda \Phi _{\lambda }^{A}=-\lambda & \text{in 
}A \\ 
\Phi _{\lambda }^{A}=0 & \text{on }\partial A.%
\end{array}%
\right.
\end{equation*}
It follows that the function $u:=1-\Phi _{\lambda }^{A}$ solves the boundary
value problem%
\begin{equation*}
\left\{ 
\begin{array}{ll}
\Delta u-\lambda u=0 & \text{in }A \\ 
u=1 & \text{on }\partial A.%
\end{array}%
\right.
\end{equation*}%
Note that $\gamma _{\lambda }-\gamma _{\lambda }^{A}=G_{\lambda
}1_{K}-G_{\lambda }^{A}1_{K}$ is $\lambda $-harmonic in $A$ and is equal to $%
\gamma _{\lambda }$ on $\partial A$, which implies by the comparison
principle in $A$ that%
\begin{equation*}
\gamma _{\lambda }-\gamma _{\lambda }^{A}\leq (\sup_{\partial A}\gamma
_{\lambda })u\ \ \ \text{in }A,
\end{equation*}%
which proves (\ref{gaA}).

$\left( b\right) $ Set $\Omega =K^{c}$. As in $\left( a\right) $, the
function $u:=1-\Phi _{\lambda }^{\Omega }$ solves the following boundary
value problem:%
\begin{equation*}
\left\{ 
\begin{array}{ll}
\Delta u-\lambda u=0 & \text{in }\Omega \\ 
u=1 & \text{on }\partial \Omega%
\end{array}%
\right.
\end{equation*}%
The function $\gamma _{\lambda }=G_{\lambda }1_{K}$ is non-negative, $%
\lambda $-harmonic, and minimal in $\Omega $. On $\partial \Omega =\partial
K $ we have%
\begin{equation}
\gamma _{\lambda }\leq \sup_{\partial K}\gamma _{\lambda }=(\sup_{\partial
K}\gamma _{\lambda })u.  \label{gau}
\end{equation}%
Since $u$ is non-negative and $\lambda $-harmonic in $\Omega ,$ it follows
by the comparison principle in $\Omega $ that (\ref{gau}) holds also in $%
\Omega $, which proves (\ref{gaKc}).
\end{proof}

Set 
\begin{equation}
\Psi _{\lambda }^{\Omega }:=G_{\lambda }^{\Omega }\left( 1-\Phi _{\lambda
}^{\Omega }\right)  \label{Psidef}
\end{equation}%
and observe that $\Psi _{\lambda }^{\Omega }\geq 0$ by (\ref{Fi1}).

\begin{lemma}
\label{Lemma of Psi} Assume that $M$ is parabolic. Then we have the
following identity for all $x\in \Omega $:%
\begin{equation}
\Psi _{\lambda }^{\Omega }(x)=\int_{0}^{\infty }te^{-\lambda t}\partial _{t}%
\mathbb{P}_{x}(\tau _{\Omega }\leq t)dt.  \label{Psi}
\end{equation}
\end{lemma}

\begin{proof}
Integrating by parts in (\ref{Fi}) together with the parabolicity of $M$, we
obtain%
\begin{eqnarray}
\Phi _{\lambda }^{\Omega }(x) &=&-\int_{0}^{\infty }\mathbb{P}_{x}(\tau
_{\Omega }>t)de^{-\lambda t}=1+\int_{0}^{\infty }e^{-\lambda t}\partial _{t}%
\mathbb{P}_{x}(\tau _{\Omega }>t)  \notag \\
&=&1-\int_{0}^{\infty }e^{-\lambda t}\partial _{t}\mathbb{P}_{x}(\tau
_{\Omega }\leq t).  \label{1-i}
\end{eqnarray}%
On the other hand, we have%
\begin{eqnarray*}
\Psi _{\lambda }^{\Omega } &=&G_{\lambda }^{\Omega }1-G_{\lambda }^{\Omega
}\Phi _{\lambda }^{\Omega }=G_{\lambda }^{\Omega }1-\lambda G_{\lambda
}^{\Omega }G_{\lambda }^{\Omega }1 \\
&=&G_{\lambda }^{\Omega }1+\lambda \frac{\partial }{\partial \lambda }%
G_{\lambda }^{\Omega }1=\frac{\partial }{\partial \lambda }\left( \lambda
G_{\lambda }^{\Omega }1\right) =\frac{\partial }{\partial \lambda }\Phi
_{\lambda }^{\Omega }.
\end{eqnarray*}%
Hence, differentiating (\ref{1-i}) in $\lambda $, we obtain (\ref{Psi}).
\end{proof}

\subsection{Some local estimates}

Recall that, for any open set $A$ containing $K$, we have defined 
\begin{equation*}
\gamma _{\lambda }^{A}(x)=G_{\lambda }^{A}1_{K}\left( x\right)
=\int_{K}\int_{0}^{\infty }e^{-\lambda t}p_{A}(t,x,z)dtd\mu (z).
\end{equation*}%
Set also 
\begin{equation}
\dot{\gamma}_{\lambda }^{A}(x):=G_{\lambda }^{A}\gamma _{\lambda }^{A}\left(
x\right) =-\frac{\partial }{\partial \lambda }\gamma _{\lambda
}^{A}(x)=\int_{K}\int_{0}^{\infty }te^{-\lambda t}p_{A}(t,x,z)dtd\mu (z).
\label{GAdot}
\end{equation}%
Note that $\gamma _{\lambda }^{A}$ and $\dot{\gamma}_{\lambda }^{A}$ vanish
outside $A$. Note also that $\gamma _{\lambda }=\gamma _{\lambda }^{M}$ and $%
\dot{\gamma}_{\lambda }=\dot{\gamma}_{\lambda }^{M}$.

In what follows we fix a precompact open set $A\supset K$ with smooth
boundary.

\begin{lemma}
\label{Lemma iii}There exists a positive constant $C=C\left( A\right) $ such
that, for all $\lambda >0$, 
\begin{equation}
\sup_{A}\gamma _{\lambda }^{A}\leq C,  \label{GC}
\end{equation}%
\begin{equation}
\sup_{A}\dot{\gamma}_{\lambda }^{A}\leq C^{2},  \label{estimate of G^A1}
\end{equation}%
and%
\begin{equation}
\sup_{A}\Psi _{\lambda }^{A}\leq C.  \label{estimate of R_A^1}
\end{equation}
\end{lemma}

\begin{proof}
It follows from (\ref{GA}) that 
\begin{equation*}
\gamma _{\lambda }^{A}\left( x\right) \leq \int_{A}\int_{0}^{\infty
}p_{A}(t,x,z)dt\,d\mu (z)=\int_{A}g^{A}\left( x,z\right) d\mu \left(
z\right) ,
\end{equation*}%
where $g^{A}=g_{0}^{A}$ is the Green function of $\Delta $ in $A$. The
function%
\begin{equation*}
u\left( x\right) =\int_{A}g^{A}(x,z)d\mu (z)
\end{equation*}%
solves the following boundary value problem%
\begin{equation*}
\left\{ 
\begin{array}{ll}
\Delta u=-1 & \text{in }A, \\ 
u=0 & \text{on }\partial A,%
\end{array}%
\right.
\end{equation*}%
which implies that $u\left( x\right) $ is bounded. Hence, (\ref{GC}) holds
with $C=\sup u$.

By (\ref{GAdot}) we have%
\begin{equation*}
\dot{\gamma}_{\lambda }^{A}\left( x\right) =\int_{A}g_{\lambda }^{A}\left(
x,z\right) \gamma _{\lambda }^{A}\left( z\right) d\mu \left( z\right) ,
\end{equation*}%
which implies by (\ref{GC}), for any $x\in A$, 
\begin{equation*}
\dot{\gamma}_{\lambda }^{A}\left( x\right) \leq \sup_{A}\gamma _{\lambda
}^{A}\int_{A}g^{A}\left( x,z\right) d\mu \left( z\right) \leq C\sup u=C^{2},
\end{equation*}%
which proves (\ref{estimate of G^A1}).

Finally, it follows from (\ref{Psidef}) that%
\begin{equation*}
\Psi _{\lambda }^{A}\left( x\right) \leq G_{\lambda }^{A}1\left( x\right)
=\int_{A}g^{A}\left( x,z\right) d\mu \left( z\right) \leq C,
\end{equation*}%
which proves (\ref{estimate of R_A^1}).
\end{proof}

\subsection{Global estimates of $\Phi _{\protect\lambda }^{\Omega }$ and $%
\Psi _{\protect\lambda }^{\Omega }$}

So far we have used a compact set $K$ and a precompact open set $A\supset K$%
. We have also assume that $K$ and $A$ have smooth boundaries.

%%%%%%%%%%%%%%%%%%%%%%%%%%%%%%%%%%%%%%%%%%%%%%%%%%%%
In the next Lemma we estimate $\inf_{\partial A}\Phi _{\lambda }^{K^{c}} $
from below using additional geometric assumptions. Denote by $K_{\epsilon }$
the $\epsilon $-neighborhood of $K$. We will assume in addition that $%
K_{\epsilon }\subset A$ for some large enough $\epsilon $ specified below.

\begin{lemma}
\label{lemma estimate of important integral} Let $M$ be a geodesically
complete, non-compact parabolic manifold satisfying $($\ref{LY type}$)$, $%
\left( RCA\right) $. Fix a reference point $o\in K$ and set $V(r)=V(o,r)$.
Assume in addition that $K_{\epsilon }\subset A$ for sufficiently large $%
\epsilon =\epsilon \left( K\right) >0$. Then there exists a constant $c>0$
such that 
\begin{equation}
\inf_{\partial A}\Phi _{\lambda }^{K^{c}}\geq c\int_{(\mathrm{diam}%
A)^{2}}^{\infty }(1-e^{-\lambda s})\frac{1}{V(\sqrt{s})H(\sqrt{s})^{2}}ds,
\label{resolvent dirichlet}
\end{equation}%
where%
\begin{equation}
H(r):=1+\left( \int_{1}^{r}\frac{s}{V(s)}ds\right) _{+}.  \label{H}
\end{equation}%
In addition, we have:

\begin{enumerate}
\item[$\left( i\right) $] if $V\left( r\right) $ is subcritical then, for $%
0<\lambda \leq \frac{1}{(\mathrm{diam}A)^{2}}$, 
\begin{equation}
\inf_{\partial A}\Phi _{\lambda }^{K^{c}}\geq c\lambda V(\frac{1}{\sqrt{%
\lambda }}).  \label{resolvent estimate (i)}
\end{equation}

\item[$\left( ii\right) $] If $V\left( r\right) $ is critical then, for $%
0<\lambda \leq \frac{1}{(\mathrm{diam}A)^{2}}$, 
\begin{equation}
\inf_{\partial A}\Phi _{\lambda }^{K^{c}}\geq \frac{c}{\log \frac{1}{\lambda 
}}.  \label{resolvent estimate (ii)}
\end{equation}
\end{enumerate}
\end{lemma}

%%%%%%%%%%%%%%%%%%%%%%%%%%%%%%%%%%%%%%%%%%%%%%%%%%%%%

\begin{proof}
Denote $\Omega =K^{c}$. By \cite[Theorem 4.9 and (4.23)]{G-SC Dirichlet}, if 
$\epsilon $ is big enough then, for all $a,y$ outside $K_{\epsilon /2}$ and
for all $s>0$, the following estimate holds: 
\begin{equation}
p_{\Omega }(s,x,y)\geq C\frac{D(s,x,y)}{V(x,\sqrt{s})}\exp \left( -c\frac{%
d^{2}(x,y)}{s}\right) ,  \label{pOM}
\end{equation}%
where 
\begin{equation*}
D(s,x,y)=\frac{H(|x|)H(\left\vert y\right\vert )}{\left( H(|x|)+H(\sqrt{s}%
)\right) \left( H(\left\vert y\right\vert )+H(\sqrt{s})\right) }.
\end{equation*}%
By \cite[(3.29)]{G-SC hitting}, we have, for any $x\notin K_{\epsilon }$, 
\begin{equation*}
\mathbb{P}_{x}\left( \tau _{\Omega }>t\right) \geq c\int_{t}^{\infty
}\inf_{y\in K_{\epsilon }\backslash K_{\epsilon /2}}p_{\Omega }(s,x,y)ds,
\end{equation*}%
where $c=c\left( K,\epsilon \right) >0$, which implies by (\ref{Fi}) 
\begin{align}
\Phi _{\lambda }^{\Omega }(x)\geq & c\int_{0}^{\infty }\lambda e^{-\lambda
t}\left( \int_{t}^{\infty }\inf_{y\in K_{\epsilon }\backslash K_{\epsilon
/2}}p_{\Omega }(s,x,y)ds\right) dt  \notag \\
=& c\int_{0}^{\infty }\left( \int_{0}^{s}\lambda e^{-\lambda t}\inf_{y\in
K_{\epsilon }\backslash K_{\epsilon /2}}p_{\Omega }(s,x,y)dt\right) ds 
\notag \\
=& c\int_{0}^{\infty }(1-e^{-\lambda s})\inf_{y\in K_{\epsilon }\backslash
K_{\epsilon /2}}p_{\Omega }(s,x,y)ds.  \label{change of variable}
\end{align}%
Assume that $x\in \partial A$. Since $y\in K_{\epsilon }$, we see that $%
d\left( x,y\right) \leq \mathrm{diam}A$. Also, $\left\vert x\right\vert
,\left\vert y\right\vert $ are bounded by $\mathrm{diam}A+e$. It follows
from (\ref{pOM}) that if $s\geq (\mathrm{diam}A)^{2}$ then%
\begin{equation*}
p_{\Omega }(s,x,y)\geq \frac{c}{V(\sqrt{s})H(\sqrt{s})^{2}}.
\end{equation*}%
Substituting into (\ref{change of variable}) yields (\ref{resolvent
dirichlet}). %%%%%%%%%%%%%%%%%%%%%%%%%%%%%%%%%%%%%%%%%%%%%%%%

%%%%%%%%%%%%%%%%%% Case (i) %%%%%%%%%%%%%%%%%%%%%%%%%%%%%
In the case $\left( i\right) $, when $V$ is subcritical, we obtain from (\ref%
{H}) 
\begin{equation}
H(r)\approx \frac{r^{2}}{V(r)}.  \label{Hsub}
\end{equation}%
Substituting into (\ref{resolvent dirichlet}), we obtain, for $0<\lambda
\leq \frac{1}{(\mathrm{diam}A)^{2}}$, 
\begin{equation*}
\inf_{\partial A}\Phi _{\lambda }^{\Omega }\geq c\int_{1/\lambda }^{\infty
}(1-e^{-\lambda s})\frac{V(\sqrt{s})}{s^{2}}ds\geq c\lambda V(\frac{1}{\sqrt{%
\lambda }}),
\end{equation*}%
which proves (\ref{resolvent estimate (i)}).

In the case $\left( ii\right) $, when $V$ is critical, we have 
\begin{equation}
H(r)\approx \log r,  \label{Hc}
\end{equation}%
which implies, for $0<\lambda \leq \frac{1}{(\mathrm{diam}A)^{2}}$, 
\begin{align*}
\inf_{\partial A}\Phi _{\lambda }^{\Omega }\geq & c\int_{1/\lambda }^{\infty
}(1-e^{-\lambda s})\frac{ds}{s\log ^{2}s} \\
\geq & c(1-e^{-1})\int_{1/\lambda }^{\infty }\frac{d\log s}{\log ^{2}s} \\
=& c(1-e^{-1})\frac{1}{\log \frac{1}{\lambda }},
\end{align*}%
which proves (\ref{resolvent estimate (ii)}).
\end{proof}

\begin{lemma}
\label{Lemma derivative RHS}Let $M$ be a geodesically complete, non-compact
parabolic manifold satisfying $($\ref{LY type}$)$, $\left( RCA\right) $.
Assume in addition that $K_{\epsilon }\subset A$ for sufficiently large $%
\epsilon =\epsilon \left( K\right) >0$. Assume also that $V\left( r\right)
:=V\left( o,r\right) $ is either critical or subcritical. Then there exists
a constant $C>0$ such that, for small enough $\lambda >0$, 
\begin{equation}
\sup_{\partial A}\Psi _{\lambda }^{K^{c}}\leq \frac{C}{\lambda \log ^{2}%
\frac{1}{\lambda }}.  \label{estimate of R_K^1}
\end{equation}
\end{lemma}

\begin{proof}
Set $\Omega =K^{c}$. Fix $a\in \partial A$ and set 
\begin{equation*}
T=\frac{1}{\lambda \log ^{2}\frac{1}{\lambda }}.
\end{equation*}%
In the identity (\ref{Psi}) for $\Psi _{\lambda }^{\Omega }$, let us
decompose the integration into two intervals: $\left[ 0,T\right] $ and $%
[T,\infty )$. For the first interval, we have by integration by parts%
\begin{equation*}
\int_{0}^{T}te^{-\lambda t}\partial _{t}\mathbb{P}_{a}(\tau _{\Omega }\leq
t)dt=Te^{-\lambda T}\mathbb{P}_{a}\left( \tau _{\Omega }\leq T\right)
-\int_{0}^{T}e^{-\lambda t}(1-\lambda t)\mathbb{P}_{a}(\tau _{\Omega }\leq
t)dt.
\end{equation*}%
Assume that $\lambda <e$ so that $\log ^{2}\frac{1}{\lambda }>1$ and, hence, 
$\lambda T<1$. It follows that $1-\lambda t\geq 0$ on $\left[ 0,T\right] $
and, therefore, the integral in the right hand side of the above identity is
non-negative. It follows that%
\begin{equation*}
\int_{0}^{T}te^{-\lambda t}\partial _{t}\mathbb{P}_{a}(\tau _{\Omega }\leq
t)dt\leq T,
\end{equation*}%
which matches the required estimate (\ref{estimate of R_K^1}).

Let us estimate the integral (\ref{Psi}) over $[T,\infty )$. By \cite[Remark
4.3]{G-SC hitting}, if $\epsilon $ is large enough then, for all $a\in
\partial A\subset \Omega $ and for all $t\geq t_{0}$ (where $t_{0}$ depends
on $\func{diam}A$), we have 
\begin{equation}
\partial _{t}\mathbb{P}_{a}(\tau _{\Omega }\leq t)\leq \frac{C}{V\left( 
\sqrt{t}\right) H^{2}\left( \sqrt{t}\right) },  \label{Hest}
\end{equation}%
where $H$ is defined by (\ref{H}). Assuming that $\lambda $ is so small that 
$T>t_{0}$ and using (\ref{Hest}), we obtain 
\begin{equation}
\int_{T}^{\infty }te^{-\lambda t}\partial _{t}\mathbb{P}_{a}(\tau _{\Omega
}\leq t)dt\leq C\int_{T}^{\infty }\frac{te^{-\lambda t}dt}{V\left( \sqrt{t}%
\right) H^{2}\left( \sqrt{t}\right) }.  \label{sH}
\end{equation}%
Consider first the case when $V\left( r\right) $ is critical, that is, $%
V\left( r\right) \approx r^{2}$. Then $H\left( r\right) \approx \log r$ and
we obtain%
\begin{equation*}
\int_{T}^{\infty }te^{-\lambda t}\partial _{t}\mathbb{P}_{a}(\tau _{\Omega
}\leq t)dt\leq C\int_{T}^{\infty }\frac{e^{-\lambda t}dt}{\log ^{2}t}\leq 
\frac{C}{\log ^{2}T}\int_{0}^{\infty }e^{-\lambda t}dt=\frac{C}{\lambda \log
^{2}T}.
\end{equation*}%
Taking $\lambda >0$ sufficiently small so that $\log ^{2}\frac{1}{\lambda }%
\leq \frac{1}{\sqrt{\lambda }}$, we obtain $T\geq \frac{1}{\sqrt{\lambda }}$
and $\log T\geq \frac{1}{2}\log \frac{1}{\lambda }$, whence%
\begin{equation*}
\int_{T}^{\infty }te^{-\lambda t}\partial _{t}\mathbb{P}_{a}(\tau _{\Omega
}\leq t)dt\leq 4CT,
\end{equation*}%
which proved (\ref{estimate of R_K^1}) in the critical case.

Assume now that $V\left( r\right) $ is subcritical. Then, for $r>2$, we have%
\begin{equation*}
\frac{r^{2}}{V\left( r\right) }\leq 3\int_{r/2}^{r}\frac{tdt}{V\left(
t\right) }\leq 3H\left( r\right) .
\end{equation*}%
Substituting into (\ref{sH}), we obtain 
\begin{equation*}
\int_{T}^{\infty }te^{-\lambda t}\partial _{t}\mathbb{P}_{a}(\tau _{\Omega
}\leq t)dt\leq C\int_{T}^{\infty }\frac{e^{-\lambda t}dt}{H\left( \sqrt{t}%
\right) }\leq \frac{C}{\lambda H(\sqrt{T})}\leq \frac{CV(\sqrt{T})}{\lambda T%
},
\end{equation*}%
where in the last inequality we have used (\ref{Hsub}). In order to prove
that the right hand side is bounded by $CT$, it suffices to verify that%
\begin{equation*}
V(\sqrt{T})\leq C\lambda T^{2}.
\end{equation*}%
Since $\log \frac{1}{\lambda }\approx \log T$ and, hence, $\lambda \approx 
\frac{1}{T\log ^{2}T}$, it suffices to prove that%
\begin{equation*}
V(\sqrt{T})\leq \frac{CT}{\log ^{2}T}
\end{equation*}%
for large enough $T$. Putting $T=r^{2}$, this inequality is equivalent to 
\begin{equation}
\log ^{2}r\leq C\frac{r^{2}}{V(r)}.  \label{subcritical bound critical}
\end{equation}%
Since $M$ is subcritical, there exists a constant $b>0$ such that, for large
enough $r$,%
\begin{equation}
b\leq \int_{1}^{r}\frac{tdt}{V(t)}\leq C\frac{r^{2}}{V(r)}.
\label{subcritical bound}
\end{equation}%
Since 
\begin{equation}
\int_{1}^{r}\frac{tdt}{V(t)}=\int_{1}^{r}\frac{t^{2}}{V(t)}d\log t,
\label{subcritical bound 2}
\end{equation}%
substituting (\ref{subcritical bound}) into the right hand side of (\ref%
{subcritical bound 2}), we obtain 
\begin{equation*}
\log r=\int_{1}^{r}d\log t\leq \int_{1}^{r}\frac{C}{b}\frac{t^{2}}{V(t)}%
d\log t=\int_{1}^{r}\frac{C}{b}\frac{tdt}{V(t)}\leq \frac{C^{2}}{b}\frac{%
r^{2}}{V(r)}.
\end{equation*}%
Substituting this into (\ref{subcritical bound 2}) again, we obtain for
large $r>0$, 
\begin{equation*}
\log ^{2}r=2\int_{1}^{r}\log td\log t\leq 2\int_{1}^{r}\frac{C^{2}}{b}\frac{%
t^{2}}{V(t)}d\log t\leq \frac{2C^{3}}{b}\frac{r^{2}}{V(r)},
\end{equation*}%
whence (\ref{subcritical bound critical}) follows.
\end{proof}

\section{On-diagonal estimates at center}

\setcounter{equation}{0}\label{SecOndiag}In this section we prove Theorem %
\ref{main theorem}. In order to obtain the upper bound of $p\left(
t,o,o\right) $ on $M=M_{1}\#...\#M_{k}$, we use the integrated resolvent
introduced in the previous section. This idea of using the resolvent on a
connected sum goes back to Woess \cite[p.\ 96]{Woess} where it was used in
the setting of connected sums of graphs. Implementation in the present case
of manifolds requires much more technique, though.

\subsection{Estimates of integrated resolvent on connected sums}

From now on let $M=M_{1}\#M_{2}\#\cdots \#M_{k}$ be a connected sum of
parabolic manifolds $M_{1},\ldots ,M_{k}$ with a central part $K$. Let $A$
be a connected, precompact open subset of $M$ with smooth boundary and such
that $K\subset A$. In fact, we will need that $K_{\epsilon }\subset A$ for
large enough $\epsilon $. Set 
\begin{equation*}
\partial A_{i}:=\partial A\cap E_{i},\ \ \ 1\leq i\leq k
\end{equation*}%
so that $\partial A=\sqcup _{i}\partial A_{i}$ (see Fig. \ref{figure:
connectedsum2}). 
\begin{figure}[tbph]
\begin{center}
\scalebox{0.9}{
\includegraphics{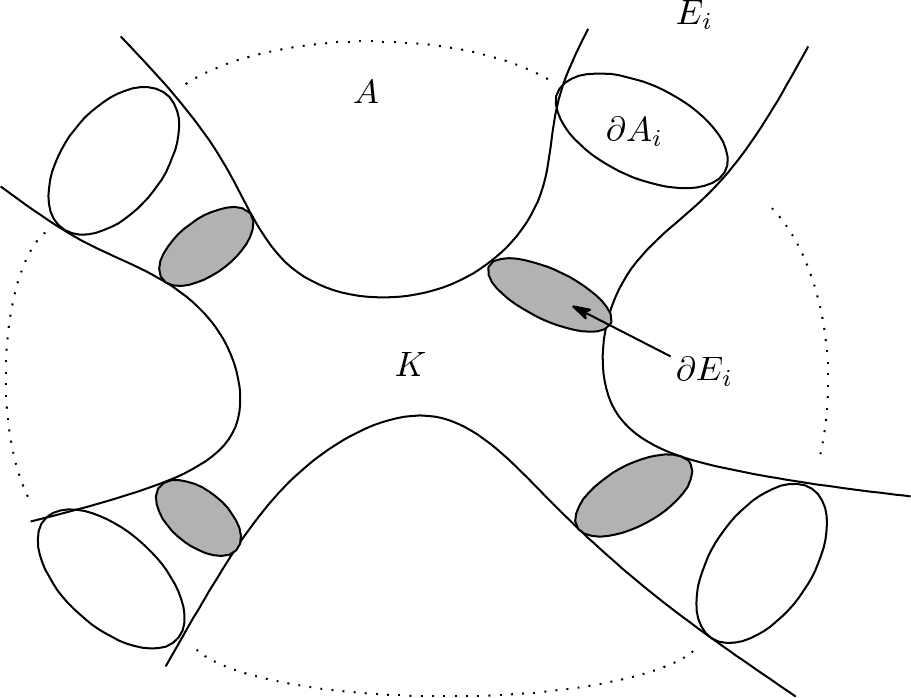} 
}
\end{center}
\caption{Sets $K$ and $A$ in the connected sum $M$.}
\label{figure: connectedsum2}
\end{figure}

\begin{lemma}
\label{Lemma estimate of G sum}There is a constant $h=h\left( A,K\right) >0$
such that, for any $\lambda >0$,%
\begin{equation}
h(\sup_{\partial K}\gamma _{\lambda })\sum_{i=1}^{k}\inf_{\partial
A_{i}}\Phi _{\lambda }^{E_{i}}\leq \sup_{\partial K}\gamma _{\lambda }^{A}.
\label{estimate of G sum (i)}
\end{equation}
\end{lemma}

\begin{proof}
As it follows from (\ref{Gla}) and (\ref{GA}) the function 
\begin{equation*}
u:=\gamma _{\lambda }-\gamma _{\lambda }^{A}=G_{\lambda }1_{K}-G_{\lambda
}^{A}1_{K}
\end{equation*}%
is $\lambda $-harmonic in $A$. Consider the function $h_{i}$ in $A$ that
solves the Dirichlet problem%
\begin{equation*}
\left\{ 
\begin{array}{ll}
\Delta h_{i}=0 & \text{in }A \\ 
h_{i}=1_{\partial A_{i}} & \text{on }\partial A.%
\end{array}%
\right.
\end{equation*}%
Since on $\partial A_{i}$ we have%
\begin{equation*}
u\leq \sup_{\partial A_{i}}\gamma _{\lambda }=(\sup_{\partial A_{i}}\gamma
_{\lambda })h_{i},
\end{equation*}%
it follows that on $\partial A$%
\begin{equation}
u\leq \sum_{i=1}^{n}(\sup_{\partial A_{i}}\gamma _{\lambda })h_{i}.
\label{uh}
\end{equation}%
Since $h_{i}$ is $\lambda $-superharmonic in $A$, we conclude by the
comparison principle in $A$ that (\ref{uh}) holds in $A$. Let us also
observe that on $\partial A$ 
\begin{equation}
\sum_{i=1}^{k}h_{i}=1,  \label{h1}
\end{equation}%
which implies then that (\ref{h1}) holds in $A$.

Since in $E_{i}$ we have $\Phi _{\lambda }^{K^{c}}=\Phi _{\lambda }^{E_{i}}$%
, we obtain by Lemma \ref{Lemga}$\left( b\right) $ that in $E_{i}$%
\begin{equation*}
\gamma _{\lambda }\leq (\sup_{\partial K}\gamma _{\lambda })(1-\Phi
_{\lambda }^{E_{i}}),
\end{equation*}%
which implies%
\begin{equation*}
\sup_{\partial A_{i}}\gamma _{\lambda }\leq (\sup_{\partial K}\gamma
_{\lambda })\sup_{\partial A_{i}}(1-\Phi _{\lambda
}^{E_{i}})=(\sup_{\partial K}\gamma _{\lambda })(1-\inf_{\partial A_{i}}\Phi
_{\lambda }^{E_{i}}).
\end{equation*}%
Substituting into (\ref{uh}) and recalling the definition of $u$, we obtain
that on $A$%
\begin{equation}
\gamma _{\lambda }\leq \gamma _{\lambda }^{A}+(\sup_{\partial K}\gamma
_{\lambda })\sum_{i=1}^{k}(1-\inf_{\partial A_{i}}\Phi _{\lambda
}^{E_{i}})h_{i}.  \label{gaga}
\end{equation}%
Let $x\in \partial K$ be a point where $\gamma _{\lambda }$ attains its
maximum on $\partial K$. Considering (\ref{gaga}) at this point $x$ we obtain%
\begin{equation*}
\gamma _{\lambda }\left( x\right) \leq \gamma _{\lambda }^{A}\left( x\right)
+\gamma _{\lambda }\left( x\right) \sum_{i=1}^{k}(1-\inf_{\partial
A_{i}}\Phi _{\lambda }^{E_{i}})h_{i}\left( x\right) ,
\end{equation*}%
whence by (\ref{h1})%
\begin{equation*}
\gamma _{\lambda }\left( x\right) \sum_{i=1}^{k}(\inf_{\partial A_{i}}\Phi
_{\lambda }^{E_{i}})h_{i}\left( x\right) \leq \gamma _{\lambda }^{A}\left(
x\right) .
\end{equation*}%
This implies (\ref{estimate of G sum (i)}) with $h:=\min_{i}\inf_{\partial
K}h_{i}>0$.
\end{proof}

\begin{lemma}
There exists a constant $h=h\left( A,K\right) >0$ such that 
\begin{equation}
h(\sup_{\partial K}\dot{\gamma}_{\lambda })\sum_{i=1}^{k}\inf_{\partial
A_{i}}\Phi _{\lambda }^{E_{i}}\leq \sup_{\partial K}\dot{\gamma}_{\lambda
}^{A}+(\sup_{\partial K}\gamma _{\lambda })\left( \sup_{\partial K}\Psi
_{\lambda }^{A}+\sum_{i=1}^{k}\sup_{\partial A_{i}}\Psi _{\lambda
}^{E_{i}}\right) .  \label{estimate of G sum (ii)}
\end{equation}
\end{lemma}

\begin{proof}
By (\ref{gadot}) and (\ref{GAdot}), the function 
\begin{equation*}
v:=\dot{\gamma}_{\lambda }-\dot{\gamma}_{\lambda }^{A}=G_{\lambda }\gamma
_{\lambda }-G_{\lambda }^{A}\gamma _{\lambda }^{A}
\end{equation*}%
solves in $A$ the following boundary value problem:%
\begin{equation*}
\left\{ 
\begin{array}{ll}
\Delta v-\lambda v=-\left( \gamma _{\lambda }-\gamma _{\lambda }^{A}\right)
& \text{in }A \\ 
v=\dot{\gamma}_{\lambda } & \text{on }\partial A.%
\end{array}%
\right.
\end{equation*}%
Consider also function $w$ that solves the problem%
\begin{equation*}
\left\{ 
\begin{array}{ll}
\Delta w-\lambda w=0 & \text{in }A \\ 
w=\dot{\gamma}_{\lambda } & \text{on }\partial A.%
\end{array}%
\right.
\end{equation*}%
Then we have%
\begin{equation}
v=G_{\lambda }^{A}\left( \gamma _{\lambda }-\gamma _{\lambda }^{A}\right) +w.
\label{vw}
\end{equation}%
Using the estimate (\ref{gaA}) of Lemma \ref{Lemga}$\left( a\right) $ and (%
\ref{Psidef}), we obtain that in $A$ 
\begin{equation}
G_{\lambda }^{A}\left( \gamma _{\lambda }-\gamma _{\lambda }^{A}\right) \leq
(\sup_{\partial A}\gamma _{\lambda })G_{\lambda }^{A}\left( 1-\Phi _{\lambda
}^{A}\right) \ =(\sup_{\partial A}\gamma _{\lambda })\Psi _{\lambda }^{A}.
\label{Gaga1}
\end{equation}%
Observe that%
\begin{equation*}
\gamma _{\lambda }\leq \sup_{\partial K}\gamma _{\lambda }\ \text{in\ }K^{c}
\end{equation*}%
because the constant function $\sup_{\partial K}\gamma _{\lambda }$ is $%
\lambda $-superharmonic in $K^{c}$, while $\gamma _{\lambda }$ is minimal $%
\lambda $-harmonic that is bounded by $\sup_{\partial K}\gamma _{\lambda }$
on $\partial K^{c}$. Hence, we obtain from (\ref{Gaga1}) that%
\begin{equation}
G_{\lambda }^{A}\left( \gamma _{\lambda }-\gamma _{\lambda }^{A}\right) \leq
(\sup_{\partial K}\gamma _{\lambda })\Psi _{\lambda }^{A}\ \ \text{in }A.
\label{Gaga}
\end{equation}%
In order to estimate $w$, let us represent this function in the form%
\begin{equation*}
w=\sum_{i=1}^{k}w_{i},
\end{equation*}%
where $w_{i}$ solves the Dirichlet problem%
\begin{equation*}
\left\{ 
\begin{array}{ll}
\Delta w_{i}-\lambda w_{i}=0 & \text{in }A \\ 
w_{i}=\dot{\gamma}_{\lambda }1_{\partial A_{i}} & \text{on }\partial A.%
\end{array}%
\right.
\end{equation*}%
Let $h_{i}$ be the same as in the proof of Lemma \ref{Lemma estimate of G
sum}. By the comparison principle, we have that in $A$%
\begin{equation}
w_{i}\leq (\sup_{\partial A_{i}}\dot{\gamma}_{\lambda })h_{i}.  \label{wi}
\end{equation}%
Let us prove further that%
\begin{equation}
\dot{\gamma}_{\lambda }-G_{\lambda }^{E_{i}}\gamma _{\lambda }\leq
(\sup_{\partial E_{i}}\dot{\gamma}_{\lambda })(1-\Phi _{\lambda }^{E_{i}})\
\ \text{in }E_{i}.  \label{Ei}
\end{equation}%
Indeed, by (\ref{gadot}), the function 
\begin{equation*}
\dot{\gamma}_{\lambda }-G_{\lambda }^{E_{i}}\gamma _{\lambda }=G_{\lambda
}\gamma _{\lambda }-G_{\lambda }^{E_{i}}\gamma _{\lambda }
\end{equation*}%
is non-negative, $\lambda $-harmonic, and minimal in $E_{i}$. Besides, it is
bounded by $\sup_{\partial E_{i}}\dot{\gamma}_{\lambda }$ on $\partial E_{i}$%
. The function $1-\Phi _{\lambda }^{E_{i}}$ is non-negative and $\lambda $%
-harmonic in $E_{i}$, and is equal to $1$ on $\partial E_{i}$. The estimate (%
\ref{Ei}) follows by the comparison principle in $E_{i}$.

Similarly, we have%
\begin{equation*}
\gamma _{\lambda }\leq (\sup_{\partial E_{i}}\gamma _{\lambda })(1-\Phi
_{\lambda }^{E_{i}})\ \text{in }E_{i},
\end{equation*}%
because $\gamma _{\lambda }$ is non-negative, $\lambda $-harmonic and
minimal in $E_{i}$, and is bounded by $\sup_{\partial E_{i}}\gamma _{\lambda
}$ on $\partial E_{i}$. It follows that in $E_{i}$%
\begin{equation*}
G_{\lambda }^{E_{i}}\gamma _{\lambda }\leq (\sup_{\partial E_{i}}\gamma
_{\lambda })G_{\lambda }^{E_i}(1-\Phi _{\lambda }^{E_{i}})=(\sup_{\partial
E_{i}}\gamma _{\lambda })\Psi _{\lambda }^{E_{i}}.
\end{equation*}%
Combining with (\ref{Ei}), we obtain that in $E_{i}$%
\begin{equation*}
\dot{\gamma}_{\lambda }\leq (\sup_{\partial E_{i}}\dot{\gamma}_{\lambda
})(1-\Phi _{\lambda }^{E_{i}})+(\sup_{\partial E_{i}}\gamma _{\lambda })\Psi
_{\lambda }^{E_{i}}.
\end{equation*}%
Substituting into (\ref{wi}), we obtain that in $A$%
\begin{equation*}
w\leq \sum_{i=1}^{k}(\sup_{\partial A_{i}}\dot{\gamma}_{\lambda })h_{i}\leq
\sum_{i=1}^{k}\left( (\sup_{\partial E_{i}}\dot{\gamma}_{\lambda
})(1-\inf_{\partial A_{i}}\Phi _{\lambda }^{E_{i}})+(\sup_{\partial
E_{i}}\gamma _{\lambda })\sup_{\partial A_{i}}\Psi _{\lambda
}^{E_{i}}\right) h_{i}.
\end{equation*}%
Combining with (\ref{vw}) and (\ref{Gaga}), we obtain the following estimate
of the function $v=\dot{\gamma}_{\lambda }-\dot{\gamma}_{\lambda }^{A}$ in $%
A $:%
\begin{equation*}
\dot{\gamma}_{\lambda }-\dot{\gamma}_{\lambda }^{A}\leq (\sup_{\partial
K}\gamma _{\lambda })\Psi _{\lambda }^{A}+\sum_{i=1}^{k}\left(
(\sup_{\partial E_{i}}\dot{\gamma}_{\lambda })(1-\inf_{\partial A_{i}}\Phi
_{\lambda }^{E_{i}})+(\sup_{\partial E_{i}}\gamma _{\lambda })\sup_{\partial
A_{i}}\Psi _{\lambda }^{E_{i}}\right) h_{i}.
\end{equation*}%
Let $x$ be a point of maximum of $\dot{\gamma}_{\lambda }$ on $\partial K$.
It follows that%
\begin{equation*}
\dot{\gamma}_{\lambda }\left( x\right) \leq \dot{\gamma}_{\lambda
}^{A}\left( x\right) +(\sup_{\partial K}\gamma _{\lambda })\Psi _{\lambda
}^{A}\left( x\right) +\sum_{i=1}^{k}\left( \dot{\gamma}_{\lambda }\left(
x\right) (1-\inf_{\partial A_{i}}\Phi _{\lambda }^{E_{i}})+(\sup_{\partial
E_{i}}\gamma _{\lambda })\sup_{\partial A_{i}}\Psi _{\lambda
}^{E_{i}}\right) h_{i}\left( x\right) .
\end{equation*}%
Since $\sum h_{i}\equiv 1$, we see that $\dot{\gamma}_{\lambda }\left(
x\right) $ cancels out in the both sides, and we obtain 
\begin{equation*}
\dot{\gamma}_{\lambda }\left( x\right) \sum_{i=1}^{k}(\inf_{\partial
A_{i}}\Phi _{\lambda }^{E_{i}})h_{i}\left( x\right) \leq \dot{\gamma}%
_{\lambda }^{A}\left( x\right) +(\sup_{\partial K}\gamma _{\lambda })\Psi
_{\lambda }^{A}\left( x\right) +\sum_{i=1}^{k}(\sup_{\partial E_{i}}\gamma
_{\lambda })(\sup_{\partial A_{i}}\Psi _{\lambda }^{E_{i}})h_{i}\left(
x\right) .
\end{equation*}%
Since $h\leq h_{i}\left( x\right) \leq 1$ where $h:=\min_{i}\inf_{K}h_{i}>0$%
, we obtain from here (\ref{estimate of G sum (ii)}).
\end{proof}

\subsection{Proof of Theorem \protect\ref{main theorem}: Upper bound}

\label{SecUpper} As in the statement of Theorem \ref{main theorem}, let $M$
be a connected sum of parabolic manifolds $M_{1},\ldots ,M_{k}$, where all $%
M_{i}$, $i=1,\ldots ,k$ satisfy $($\ref{LY type}$)$ and $\left( RCA\right) $%
. Let $V_{i}(r)=V_{i}\left( o_{i},r\right) $ be the volume function on $%
M_{i} $ at $o_{i}\in K_{i}=M_{i}\setminus E_{i}$. We also assume that every $%
V_i(r)$ is either critical or subcritical, that is, condition (d) of Section %
\ref{notion}. Let $V(r)=V\left(o,r\right) $ be the volume function on $M$ at
a reference point $o\in K$.

It suffices to prove the main estimate (\ref{ptoo}) for large enough $t$
because for small $t$ we have $p(t,o,o)\asymp t^{-N/2}$ and $V(\sqrt{t}%
)\asymp t^{N/2}$.

Fix a connected precompact open set $A$ with smooth boundary such that $%
A\supset K_{\epsilon }$ for large enough $\epsilon >0$ as in Lemmas \ref%
{lemma estimate of important integral} and \ref{Lemma derivative RHS}
applied to all ends $M_{i}$.

Recall that the integrated resolvent $\gamma _{\lambda }$ is defined by (\ref%
{Gla}). By Lemmas \ref{Lemma iii} and \ref{Lemma estimate of G sum}, we
have, for any $\lambda >0$ and any $i=1,...,k$ 
\begin{equation}
\sup_{\partial K}\gamma _{\lambda }\leq \frac{C}{\inf_{\partial A_{i}}\Phi
_{\lambda }^{E_{i}}},  \label{estimate of G sum}
\end{equation}%
where $C=C\left( K,A\right) $.

Assume first that all manifolds $M_{i}$ are subcritical. Applying (\ref%
{resolvent estimate (i)}) on each end $M_{i}$ we obtain that%
\begin{equation*}
\inf_{\partial A_{i}}\Phi _{\lambda }^{E_{i}}\geq c\lambda V_{i}(\frac{1}{%
\sqrt{\lambda }})
\end{equation*}%
provided $\lambda \leq \lambda _{0}=\lambda _{0}\left( A\right) $.
Substituting into (\ref{estimate of G sum}), we obtain that, for $\lambda
\leq \lambda _{0}$, 
\begin{equation*}
\sup_{\partial K}\gamma _{\lambda }\leq \frac{C}{\lambda V_{\max }(\frac{1}{%
\sqrt{\lambda }})},
\end{equation*}%
where $V_{\max }(r)=\max_{1\leq i\leq k}V_{i}(r)$. By Lemma \ref{Lemma of G}$%
\left( i\right) $, we conclude that, for all $t\geq t_{0}=t_{0}\left(
\lambda _{0}\right) $, 
\begin{equation}
p(t,o,o)\leq \frac{C}{V_{\max }(\sqrt{t})}  \label{on-diagonal subcritical}
\end{equation}%
which proves the on-diagonal upper bound in (\ref{ptoo}) in the subcritical
case.

Assume now that there exists at least one critical end. Let it be $M_{j}$.
Applying (\ref{resolvent estimate (ii)}) in $M_{j}$, we have%
\begin{equation}
\inf_{\partial A}\Phi _{\lambda }^{E_{j}}\geq \frac{c}{\log \frac{1}{\lambda 
}},  \label{Ej}
\end{equation}
which together with (\ref{estimate of G sum}) yields, for all $\lambda \leq
\lambda _{0}$,%
\begin{equation}
\sup_{\partial K}\gamma _{\lambda }\leq C\log \frac{1}{\lambda }.
\label{estimate of G critical sum}
\end{equation}%
However, as we have pointed out before, in order to obtain upper bound in (%
\ref{ptoo}) in the critical case, we need some additional argument about $%
\dot{\gamma}_{\lambda }$.

For that, let us use the estimate (\ref{estimate of G sum (ii)}) of $%
\sup_{\partial K}\dot{\gamma}_{\lambda }$. Substituting into (\ref{estimate
of G sum (ii)}) the estimates (\ref{estimate of G^A1}) and (\ref{estimate of
R_A^1}), we obtain%
\begin{equation*}
(\sup_{\partial K}\dot{\gamma}_{\lambda })\inf_{\partial A_{j}}\Phi
_{\lambda }^{E_{j}}\leq C+C\sup_{\partial K}\gamma _{\lambda }\left(
1+\sum_{i=1}^{k}\sup_{\partial A_{i}}\Psi _{\lambda }^{E_{i}}\right) .
\end{equation*}%
Substituting here (\ref{Ej}), (\ref{estimate of G critical sum}), (\ref%
{estimate of R_K^1}), we obtain, for all $\lambda \leq \lambda _{0}$, 
\begin{equation*}
\sup_{\partial K}\dot{\gamma}_{\lambda }\frac{1}{\log \frac{1}{\lambda }}%
\leq C+C\log \frac{1}{\lambda }\left( 1+\frac{1}{\lambda \log ^{2}\frac{1}{%
\lambda }}\right) \leq \frac{C^{\prime }}{\lambda \log \frac{1}{\lambda }},
\end{equation*}%
which implies%
\begin{equation*}
\sup_{\partial K}\dot{\gamma}_{\lambda }\leq \frac{C}{\lambda }\text{ for
all }\lambda \leq \lambda _{0}.
\end{equation*}%
By Lemma \ref{Lemma of G} $\left( ii\right) $, we conclude that%
\begin{equation}
p(t,o,o)\leq \frac{C}{t}\quad \text{for all }t\geq t_{0}
\label{on-diagonal critical}
\end{equation}%
which finishes the proof of the upper bound in (\ref{ptoo}) in the critical
case.

\subsection{Proof of Theorem \protect\ref{main theorem}: Lower bound}

\label{lower bound} Let $M$ be a connected sum satisfying the assumption of
Theorem \ref{main theorem}. Let us observe that%
\begin{equation}
V(r)\approx V_{1}(r)+V_{2}(r)+\cdots +V_{k}(r)\approx V_{\max }(r)
\label{VVmax}
\end{equation}%
for all $r>0$. By (\ref{on-diagonal subcritical}) and (\ref{on-diagonal
critical}), we obtain that, for all $t>0,$ 
\begin{equation}
p(t,o,o)\leq \frac{C}{V(\sqrt{t})}.  \label{on-diagonal upper bound}
\end{equation}%
Since each $V_{i}\left( r\right) $ satisfies the doubling condition, so does 
$V\left( r\right) $ by (\ref{VVmax}). By \cite[Theorem 7.2]%
{Coulhon-Grigoryan}, the upper bound (\ref{on-diagonal upper bound})
together with the doubling property of $V\left( r\right) $ implies the
matching lower bound%
\begin{equation*}
p(t,o,o)\geq \frac{c}{V(\sqrt{t})}.
\end{equation*}%
Replacing here $V$ by $V_{\max }$, we finish the proof of the lower bound in
(\ref{ptoo}) and, hence, the proof of Theorem \ref{main theorem}.

\section{Off-diagonal estimates}

\setcounter{equation}{0}\label{SecOff}In this section, we prove Theorems \ref%
{T1}-\ref{T3} by combining Theorem \ref{main theorem} with some results from 
\cite{G-SC Dirichlet}, \cite{G-SC hitting} and \cite{G-SC ends}.

For any open set $\Omega $ in any weighted manifold $M$, define the \textit{%
exit probability function} in $\Omega $: for all $x\in \Omega $ and $t>0$,%
\begin{equation*}
\psi _{\Omega }\left( y,t\right) =\mathbb{P}_{x}(\tau _{\Omega }\leq t).
\end{equation*}%
Equivalently, $\psi _{\Omega }\left( x,t\right) $ is the minimal
non-negative solution of the heat equation $\partial _{t}u=\Delta u$ in $%
\Omega \times \mathbb{R}_{+}$ with the initial condition $u|_{t=0}=0$ and
the boundary condition $u|_{\partial \Omega }=1$.

We will use the abstract upper and lower off-diagonal estimates of \cite[%
Theorem 3.5]{G-SC ends} for the heat kernel $p\left( t,x,y\right) $ on an
arbitrary manifold $M$ for $x\in A$ and $y\in B$ where $A,B$ are open
subsets of $M$ such either $\overline{A}$ and $\overline{B}$ are disjoint or 
$\overline{B}\subset A$. These estimates use the exit probabilities $\psi
_{A}\left( x,t\right) $\ and $\psi _{B}\left( y,t\right) $ and their time
derivatives. Besides, they use the quantities%
\begin{equation*}
P^{+}\left( t\right) =\sup_{s\in \left[ t/4,t\right] }\sup_{z_{1}\in
\partial A,\ z_{2}\in \partial B}p\left( s,z_{1},z_{2}\right) \text{ \ and \ 
}P^{-}\left( t\right) =\inf_{s\in \left[ t/4,t\right] }\inf_{z_{1}\in
\partial A,\ z_{2}\in \partial B}p\left( s,z_{1},z_{2}\right)
\end{equation*}%
and%
\begin{equation*}
G^{+}\left( t\right) =\int_{0}^{t}\sup_{z_{1}\in \partial A,\ z_{2}\in
\partial B}p\left( s,z_{1},z_{2}\right) ds\ \text{and}\ G^{-}\left( t\right)
=\int_{0}^{t}\inf_{z_{1}\in \partial A,\ z_{2}\in \partial B}p\left(
s,z_{1},z_{2}\right) ds.
\end{equation*}%
With these notations, the estimates of \cite[Theorem 3.5]{G-SC ends} read as
follows: for all $x\in A,y\in B$ and $t>0$,%
\begin{eqnarray}
p(t,x,y) &\approx &p_{A}\left( t,x,y\right) +P^{\pm }\left( t\right) \mathbb{%
\psi }_{A}\left( x,\tilde{t}\right) \mathbb{\psi }_{B}\left( y,\tilde{t}%
\right)  \notag \\
&&+G^{\pm }\left( \tilde{t}\right) \left[ \partial _{t}\mathbb{\psi }%
_{A}\left( x,\xi \right) \psi _{B}\left( y,\tilde{t}\right) +\partial
_{t}\psi _{B}\left( y,\zeta \right) \mathbb{\psi }_{A}\left( x,\tilde{t}%
\right) \right] ,\,  \label{general full estimate}
\end{eqnarray}%
where the index \textquotedblleft $+$\textquotedblright\ is used for the
upper bound, \textquotedblleft $-$\textquotedblright\ is used for the lower
bound, $\tilde{t}=t$ for the upper bound, $\tilde{t}=\frac{1}{4}t$ for the
lower bound, $\xi $ and $\zeta $ are some values from $\left[ t/4,t\right] $
that may be different for upper and lower bounds.

\begin{proof}[Proof of Theorem \protect\ref{T1}]
Recall that $M$ is a connected sum of $M_{1},\ldots ,M_{k}$ with a central
part $K$, where each $M_{i}$ satisfies conditions $\left( a\right) $-$\left(
d\right) $ in Subsection \ref{notion}. We apply (\ref{general full estimate}%
) with $A=E_{i}$ and $B=E_{j}$ where $i\neq j$. Since $A$ and $B$ are
disjoint, we have $p_{A}\left( t,x,y\right) =0$ for all $x\in A$ and $y\in B$%
.

Note that, for all $z_{1}\in \partial E_{i}$ and $z_{2}\in \partial E_{j}$,
the distance $d\left( z_{1},z_{2}\right) $ is bounded from above and below
by positive constants. Therefore, assuming $t>1$, we obtain by the local
Harnack inequality and Theorem \ref{main theorem} that%
\begin{equation}
P^{\pm }\left( t\right) \asymp Cp\left( ct,o,o\right) \approx \frac{1}{%
V\left( \sqrt{t}\right) }.  \label{Ppm}
\end{equation}%
Let us estimate similarly $G^{\pm }\left( t\right) $. Assuming $t>1$, we can
split the integrals in the definition of $G^{\pm }\left( t\right) $ into the
sum of two integrals: over $(0,1]$ and over $(1,t]$. The first integral is
bounded, while in the second integral we can apply the local Harnack
inequality to the heat kernel and, hence, replace $z_{1},z_{2}$ by $o$.
Using further the estimate (\ref{ptoo}) of Theorem \ref{main theorem}, we
obtain that, for large $t$,%
\begin{equation}
G^{\pm }\left( t\right) \approx \int_{1}^{t}\frac{1}{V\left( \sqrt{s}\right) 
}ds.  \label{Gint}
\end{equation}%
If all ends are subcritical, then by (\ref{subcritical}) we have, for large $%
t$, 
\begin{equation*}
\int_{1}^{t}\frac{ds}{V(\sqrt{s})}\leq \frac{Ct}{V(\sqrt{t})}.
\end{equation*}%
Since also%
\begin{equation*}
\int_{1}^{t}\frac{ds}{V(\sqrt{s})}\geq \int_{t/2}^{t}\frac{ds}{V(\sqrt{s})}%
\geq \frac{t}{2V(\sqrt{t})},
\end{equation*}%
we obtain that 
\begin{equation}
G^{\pm }\left( \tilde{t}\right) \approx \frac{t}{V(\sqrt{t})}.
\label{central integral subcritical}
\end{equation}%
If there exists at least one critical end, then $V\left( \sqrt{t}\right)
\approx t$, and (\ref{Gint}) implies, for large $t$,%
\begin{equation}
G^{\pm }\left( \tilde{t}\right) \approx \log t.
\label{central integral critical}
\end{equation}%
Note that the exit probability $\psi _{i}\left( x,t\right) $ depends only on
the intrinsic geometry of $E_{i}$. Since each $M_{i}$ satisfies $($\ref{LY
type}$)$ and $\left( RCA\right) $, we can use the results of \cite[Theorem
4.6]{G-SC hitting} that gives the following: for all $x\in E_{i}$ with large
enough $\left\vert x\right\vert $,%
\begin{equation}
\psi _{E_{i}}\left( x,t\right) \asymp \left\{ 
\begin{array}{ll}
\frac{C\left\vert x\right\vert ^{2}\exp \left( -b\left\vert x\right\vert
^{2}/t\right) }{V_i\left( \left\vert x\right\vert \right) H\left( \left\vert
x\right\vert \right) } & t<2\left\vert x\right\vert ^{2}, \\ 
\frac{C}{H\left( \sqrt{t}\right) }\int_{\left\vert x\right\vert }^{\sqrt{t}}%
\frac{sds}{V_i\left( s\right) }, & t\geq 2\left\vert x\right\vert ^{2}%
\end{array}%
\right.  \label{psiEi}
\end{equation}%
and, for large enough $\left\vert x\right\vert $ and $t$, 
\begin{equation}
\partial _{t}\mathbb{\psi }_{E_{i}}\left( x,t\right) \asymp \frac{CH\left(
\left\vert x\right\vert \right) \exp \left( -b\left\vert x\right\vert
^{2}/t\right) }{V_{i}\left( \sqrt{t}\right) \left( H\left( \left\vert
x\right\vert \right) +H\left( \sqrt{t}\right) \right) H\left( \sqrt{t}%
\right) },  \label{PsiEider}
\end{equation}%
where $H$ is the function defined in (\ref{H}). Note that in the case of
bounded $\left\vert x\right\vert $ the estimate (\ref{PsiEider}) matches the
estimate (\ref{Hest}) used in the proof of Lemma \ref{Lemma derivative RHS}.

If $M_{i}$ is subcritical then $H\left( r\right) \approx r^{2}/V_{i}\left(
r\right) $. Substituting this into then (\ref{psiEi}) and (\ref{PsiEider}),
we obtain, for all large enough $t$ and $\left\vert x\right\vert $,%
\begin{align}
\mathbb{\psi }_{E_{i}}\left( x,t\right) \asymp & Ce^{-b\frac{\left\vert
x\right\vert ^{2}}{t}},  \label{hitting subcritical} \\
\partial _{t}\mathbb{\psi }_{E_{i}}\left( x,t\right) \asymp & \frac{C}{t}%
D(x,t)e^{-b\frac{\left\vert x\right\vert ^{2}}{t}},
\label{hitting derivative subcritical}
\end{align}%
where $D$ is defined in (\ref{function D}).

If $M_{i}$ is critical then $H\left( r\right) \approx \log r$ which yields%
\begin{align}
\mathbb{\psi }_{E_{i}}\left( x,t\right) \asymp & CU(x,t)e^{-b\frac{%
\left\vert x\right\vert ^{2}}{t}},  \label{hitting critical} \\
\partial _{t}\mathbb{\psi }_{E_{i}}\left( x,t\right) \asymp & \frac{C}{t\log
t}W(x,t)e^{-b\frac{\left\vert x\right\vert ^{2}}{t}},
\label{hitting derivative critical}
\end{align}%
where $U$ is defined in (\ref{function U}) and $W$ is defined in (\ref%
{function W}).

Now we are in position to verify all the heat kernel estimates claimed in
Theorem \ref{T1} for $x\in E_{i},y\in E_{j}$ with $i\neq j$. It suffices to
prove all the estimates for large enough $\,\left\vert x\right\vert
,\left\vert y\right\vert $ and $t$. Then the estimates for all $x\in E_{i}$
and $y\in E_{j}$ (while $t$ is still large enough) follow by application of
the local Harnack inequality.

$\left( i\right) $ If all ends are subcritical, then (\ref{general full
estimate}), (\ref{Ppm}), (\ref{central integral subcritical}), (\ref{hitting
subcritical}), (\ref{hitting derivative subcritical}) yield:%
\begin{equation*}
p(t,x,y)\asymp \frac{C}{V(\sqrt{t})}\left[ 1+D(x,t)+D(y,t)\right] e^{-b\frac{%
\left\vert x\right\vert ^{2}+\left\vert y\right\vert ^{2}}{t}}.\,
\end{equation*}%
Observing that that by (\ref{function D}) $D\left( x,t\right) $ is bounded
and that%
\begin{equation*}
\left\vert x\right\vert ^{2}+\left\vert y\right\vert ^{2}\approx d^{2}\left(
x,y\right)
\end{equation*}%
we obtain (\ref{T1i}).

$\left( ii\right) $ Now let at least one of the ends be critical, so that $%
V\left( r\right) \approx r^{2}$.

$\left( ii\right) _{1}$ Let $M_{i},M_{j}$ are subcritical, then (\ref%
{general full estimate}), (\ref{Ppm}), (\ref{central integral critical}), (%
\ref{hitting subcritical}), (\ref{hitting derivative subcritical}) yield:%
\begin{equation*}
p(t,x,y)\asymp \frac{C}{t}\left( 1+\left( D(x,t)+D(y,t)\right) \log t\right)
e^{-b\frac{\left\vert x\right\vert ^{2}+\left\vert y\right\vert ^{2}}{t}},\,
\end{equation*}%
which proves (\ref{T1ii1}).

$\left( ii\right) _{2}$ Let both $M_{i}$ and $M_{j}$ be critical. Then we
obtain from (\ref{general full estimate}), (\ref{Ppm}), (\ref{central
integral critical}), (\ref{hitting critical}), (\ref{hitting derivative
critical}) that%
\begin{equation*}
p(t,x,y)\asymp \frac{C}{t}\left[ U(x,t)U(y,t)+W(x,t)U\left( y,t\right)
+U(x,t)W\left( y,t\right) \right] e^{-b\frac{\left\vert x\right\vert
^{2}+\left\vert y\right\vert ^{2}}{t}},\,
\end{equation*}%
that is, (\ref{T1ii2}).

$\left( ii\right) _{3}$ Let $M_{i}$ be subcritical and $M_{j}$ be critical.
Then we obtain similarly%
\begin{equation*}
p(t,x,y)\asymp \frac{C}{t}\left[ U\left( x,t\right) +D(x,t)U\left(
y,t\right) \log t+W(x,t)\right] e^{-b\frac{\left\vert x\right\vert
^{2}+\left\vert y\right\vert ^{2}}{t}}.
\end{equation*}%
By (\ref{U+W}) we can replace here $U+W$ by $1$, which yields (\ref{T1ii3}).
\end{proof}

For the proof of Theorems \ref{T2} and \ref{T3}, we will use again the
estimate (\ref{general full estimate}) but this time we take $A=E_{i}$ and $%
B=E_{i}^{\prime }$ where $E_{i}^{\prime }=E_{i}\setminus K^{\prime }$ and $%
K^{\prime }$ is a closed $\epsilon $-neighborhood of $K$ for large enough $%
\epsilon $. In this case we have $\overline{B}\subset A$.

Note that, for all $z_{1}\in \partial E_{i}$ and $z_{2}\in \partial
E_{i}^{\prime }$, the distance $d\left( z_{1},z_{2}\right) $ is bounded from
above and below by positive constants. Hence, arguing as above, we obtain
the same estimates of $P^{\pm }\left( t\right) ,G^{\pm }\left( t\right) $ as
stated in the proof of Theorem \ref{T1}. The estimates of $\psi _{E_{i}}$
and $\partial _{t}\psi _{E_{i}}$ also remain the same. Clearly, $\psi
_{E_{i}^{\prime }}$ and $\partial _{t}\psi _{E_{i}^{\prime }}$ satisfy
similar estimates.

To handle the term $p_{A}\left( t,x,y\right) =p_{E_{i}}(t,x,y)$ in (\ref%
{general full estimate}), we use the result of \cite[Theorem 4.9]{G-SC
Dirichlet} that says the following: for all $t>0$ and all $x,y\in E_{i}$
with large $\left\vert x\right\vert ,\left\vert y\right\vert $, 
\begin{equation*}
p_{E_{i}}(t,x,y)\asymp \frac{C}{V_{i}(x,\sqrt{t})}\left( \frac{H(\left\vert
x\right\vert )}{H(\left\vert x\right\vert )+H(\sqrt{t})}\right) \left( \frac{%
H(\left\vert y\right\vert )}{H(\left\vert y\right\vert )+H(\sqrt{t})}\right)
e^{-b\frac{d^{2}}{t}},
\end{equation*}%
where $d=d\left( x,y\right) $. If $M_{i}$ is subcritical, then $H(r)\approx
r^{2}/V(r)$, which gives 
\begin{equation}
p_{E_{i}}(t,x,y)\asymp C\frac{D(x,t)D(y,t)}{V_{i}(x,\sqrt{t})}e^{-b\frac{%
d^{2}}{t}}.  \label{Dirichlet subcritical}
\end{equation}%
If $M_{i}$ is critical, then $H(r)\approx \log r$, which gives 
\begin{equation}
p_{E_{i}}(t,x,y)\asymp C\frac{W(x,t)W(y,t)}{V_{i}(x,\sqrt{t})}e^{-b\frac{%
d^{2}}{t}}.  \label{Dirichlet critical}
\end{equation}

For the proof of Theorems \ref{T2} and \ref{T3} we need the following lemma.

\begin{lemma}
\label{lemma e} For all $x,y\in E_{i}$ and $\sqrt{t}\geq \min (|x|,|y|)$ we
have%
\begin{equation}
Ce^{-b\frac{|x|^{2}+|y|^{2}}{t}}\asymp C^{\prime }e^{-b^{\prime }\frac{%
d^{2}(x,y)}{t}}.  \label{equivalence of e}
\end{equation}%
Moreover, if $\sqrt{t}\geq \left\vert x\right\vert $ then%
\begin{equation}
\frac{C}{V_{i}\left( x,\sqrt{t}\right) }e^{-b\frac{d^{2}(x,y)}{t}}\asymp 
\frac{C^{\prime }}{V_{i}\left( \sqrt{t}\right) }e^{-b^{\prime }\frac{%
d^{2}(x,y)}{t}}  \label{Vixo}
\end{equation}
\end{lemma}

\begin{proof}
Set $\delta =\mathrm{diam}K$. The triangle inequality $\left\vert
x\right\vert +\left\vert y\right\vert +\delta \geq d(x,y)$ implies%
\begin{equation}
e^{-b\frac{|x|^{2}+|y|^{2}}{t}}\leq e^{-b^{\prime }\frac{d^{2}(x,y)-\delta
^{2}}{t}}\leq C^{\prime }e^{-b^{\prime }\frac{d^{2}(x,y)}{t}}.
\label{upper e}
\end{equation}%
To prove the opposite inequality, assume that $\left\vert x\right\vert \leq 
\sqrt{t}$ (the case $\left\vert y\right\vert \leq \sqrt{t}$ is similar). The
triangle inequality%
\begin{equation*}
\left\vert y\right\vert \leq \left\vert x\right\vert +\delta +d(x,y)
\end{equation*}%
implies%
\begin{equation*}
\left\vert x\right\vert +\left\vert y\right\vert \leq 2\left\vert
x\right\vert +\delta +d(x,y)\leq 2\sqrt{t}+\delta +d(x,y),
\end{equation*}%
whence it follows that%
\begin{equation*}
\frac{\left\vert x\right\vert ^{2}+\left\vert y\right\vert ^{2}}{t}\leq
b^{\prime }\frac{d^{2}(x,y)}{t}+const,
\end{equation*}%
which completes the proof of (\ref{equivalence of e}).

To prove (\ref{Vixo}) observe first that by (\ref{equivalence of e}), the
term $d^{2}\left( x,y\right) $ in the both sides of (\ref{Vixo}) can be
replaced by $\left\vert x\right\vert ^{2}+\left\vert y\right\vert ^{2}$. The
doubling property of\ $V_{i}\left( x,r\right) $ yields%
\begin{equation*}
\frac{V_{i}\left( o_{i},\sqrt{t}\right) }{V_{i}\left( x,\sqrt{t}\right) }%
\leq C\left( 1+\frac{\left\vert x\right\vert }{\sqrt{t}}\right) ^{\beta
}\leq Ce^{\varepsilon \frac{\left\vert x\right\vert ^{2}}{t}},
\end{equation*}%
for arbitrarily small $\varepsilon >0$, which implies that%
\begin{eqnarray}
\frac{C}{V_{i}\left( x,\sqrt{t}\right) }e^{-b\frac{\left\vert x\right\vert
^{2}+\left\vert y\right\vert ^{2}}{t}} &\leq &\frac{C^{\prime }}{V_{i}\left(
o,\sqrt{t}\right) }e^{\varepsilon \frac{\left\vert x\right\vert ^{2}}{t}%
}e^{-b\frac{\left\vert x\right\vert ^{2}+\left\vert y\right\vert ^{2}}{t}} 
\notag \\
&\leq &\frac{C^{\prime }}{V_{i}\left( o,\sqrt{t}\right) }e^{-b^{\prime }%
\frac{\left\vert x\right\vert ^{2}+\left\vert y\right\vert ^{2}}{t}}.
\label{Vo}
\end{eqnarray}%
The opposite inequality is proved similarly.
\end{proof}

\begin{proof}[Proof of Theorem \protect\ref{T2}$\left( a\right) $]
We consider the same cases as in Theorem \ref{T1} and use the same estimates
of all the terms in (\ref{general full estimate}), except for the Dirichlet
heat kernel. Note that the case $\left( ii\right) _{3}$ cannot occur because
\thinspace $x,y$ are at the same end $E_{i}$.

$\left( i\right) $ Assume that all ends are subcritical. Substituting (\ref%
{Dirichlet subcritical}), (\ref{Ppm}), (\ref{central integral subcritical}),
(\ref{hitting subcritical}) and (\ref{hitting derivative subcritical}) into (%
\ref{general full estimate}), we obtain 
\begin{align}
p(t,x,y)\asymp & C\frac{D(x,t)D(y,t)}{V_{i}(x,\sqrt{t})}e^{-b\frac{d^{2}}{t}}
\notag \\
& +\frac{C}{V(\sqrt{t})}\left( 1+D(x,t)+D(y,t)\right) e^{-b\frac{\left\vert
x\right\vert ^{2}+\left\vert y\right\vert ^{2}}{t}}.  \label{D x+y}
\end{align}%
By (\ref{function D}) and the assumption $\sqrt{t}\leq \min \left(
\left\vert x\right\vert ,\left\vert y\right\vert \right) $ we have 
\begin{equation*}
D\left( x,t\right) =D\left( y,t\right) =1
\end{equation*}%
and, hence, 
\begin{equation}
p\left( t,x,y\right) \asymp \frac{C}{V_{i}(x,\sqrt{t})}e^{-b\frac{%
d^{2}\left( x,y\right) }{t}}+\frac{C}{V\left( \sqrt{t}\right) }e^{-b\frac{%
\left\vert x\right\vert ^{2}+\left\vert y\right\vert ^{2}}{t}}.  \label{x+y}
\end{equation}%
Using the volume doubling property of $V_{i}$, we obtain 
\begin{align}
\frac{1}{V(\sqrt{t})}e^{-b\frac{\left\vert x\right\vert ^{2}+\left\vert
y\right\vert ^{2}}{t}}=& \frac{V_{i}(o_{i,}\sqrt{t})}{V_{\max }(\sqrt{t})}%
\frac{V_{i}(x,\sqrt{t})}{V_{i}(o_{i},\sqrt{t})}\frac{1}{V_{i}(x,\sqrt{t})}%
e^{-b\frac{\left\vert x\right\vert ^{2}+\left\vert y\right\vert ^{2}}{t}} 
\notag \\
\leq & C\left( 1+\frac{\left\vert x\right\vert }{\sqrt{t}}\right) ^{\beta }%
\frac{1}{V_{i}(x,\sqrt{t})}e^{-b\frac{\left\vert x\right\vert
^{2}+\left\vert y\right\vert ^{2}}{t}}  \notag \\
\leq & \frac{C^{\prime }}{V_{i}(x,\sqrt{t})}e^{-b^{\prime }\frac{d^{2}\left(
x,y\right) }{t}},  \label{vd max}
\end{align}%
which shows that the first term in (\ref{x+y}) is dominant, hence yielding (%
\ref{Vi}).

$\left( ii\right) $ Let at least one of the ends be critical.

$\left( ii\right) _{1}$ Let $M_{i}$ be subcritical. In this case we have as
above%
\begin{equation}
p(t,x,y)\asymp \frac{C}{V_{i}(x,\sqrt{t})}e^{-b\frac{d^{2}}{t}}+C\frac{\log t%
}{t}e^{-b\frac{\left\vert x\right\vert ^{2}+\left\vert y\right\vert ^{2}}{t}%
}.  \label{caseii1 mod}
\end{equation}%
By (\ref{or2logr}) and the volume doubling property of $M_{i}$, we obtain 
\begin{align}
\frac{\log t}{t}e^{-b\frac{\left\vert x\right\vert ^{2}+\left\vert
y\right\vert ^{2}}{t}}=& \frac{\log t}{t}V_{i}(o_{i},\sqrt{t})\frac{1}{%
V_{i}(x,\sqrt{t})}\frac{V_{i}(x,\sqrt{t})}{V_{i}(o_{i},\sqrt{t})}e^{-b\frac{%
\left\vert x\right\vert ^{2}+\left\vert y\right\vert ^{2}}{t}}  \notag \\
\leq & \frac{C}{V_{i}(x,\sqrt{t})}\left( 1+\frac{|x|}{\sqrt{t}}\right)
^{\beta }e^{-b\frac{\left\vert x\right\vert ^{2}+\left\vert y\right\vert ^{2}%
}{t}}  \notag \\
\leq & \frac{C^{\prime }}{V_{i}(x,\sqrt{t})}e^{-b^{\prime }\frac{d^{2}\left(
x,y\right) }{t}}.  \label{vd log}
\end{align}%
Substituting (\ref{vd log}) into (\ref{caseii1 mod}), we obtain (\ref{Vi}).

$\left( ii\right) _{2}$ Let $M_{i}$ be critical. Substituting (\ref%
{Dirichlet critical}), (\ref{Ppm}), (\ref{central integral critical}), (\ref%
{hitting critical}) and (\ref{hitting derivative critical}) into (\ref%
{general full estimate}), we obtain 
\begin{eqnarray}
p(t,x,y) &\asymp &C\frac{W(x,t)W(y,t)}{V_{i}(x,\sqrt{t})}e^{-b\frac{d^{2}}{t}%
}  \notag \\
&&+\frac{C}{t}\left[ U(x,t)U(y,t)+W(x,t)U\left( y,t\right) +W(y,t)U(x,t)%
\right] e^{-b\frac{\left\vert x\right\vert ^{2}+\left\vert y\right\vert ^{2}%
}{t}}.  \label{same end ii2}
\end{eqnarray}%
By (\ref{function W}) and $\sqrt{t}\leq \min \left( \left\vert x\right\vert
,\left\vert y\right\vert \right) $, we have 
\begin{equation*}
W(x,t)=W(y,t)=1.
\end{equation*}%
Substituting into (\ref{same end ii2}) we obtain 
\begin{eqnarray*}
p(t,x,y) &\asymp &\frac{C}{V_{i}(x,\sqrt{t})}e^{-b\frac{d^{2}}{t}} \\
&&+\frac{C}{t}\left[ U(x,t)U(y,t)+U\left( y,t\right) +U(x,t)\right] e^{-b%
\frac{\left\vert x\right\vert ^{2}+\left\vert y\right\vert ^{2}}{t}}.
\end{eqnarray*}%
Since $U$ is bounded, (\ref{vd max}) implies that the second term is
dominated by the first one, which yields (\ref{Vi}).
\end{proof}

\begin{proof}[Proof of Theorem \protect\ref{T2}$\left( b\right) $]
Let $V_{i}\left( r\right) \approx V_{\max }\left( r\right) $. In the view of
part $\left( a\right) $, we can assume that $\sqrt{t}>\min \left( \left\vert
x\right\vert ,\left\vert y\right\vert \right) $. Since by the doubling
property of $V_{i}$%
\begin{equation*}
\frac{C}{V_{i}\left( x,\sqrt{t}\right) }e^{-b\frac{d^{2}\left( x,y\right) }{t%
}}\asymp \frac{C^{\prime }}{V_{i}\left( y,\sqrt{t}\right) }e^{-b^{\prime }%
\frac{d^{2}\left( x,y\right) }{t}}
\end{equation*}%
(cf. (\ref{Vo})), the estimate (\ref{Vi}) is symmetric in $x,y$. Hence, we
can assume that $\sqrt{t}>\left\vert x\right\vert $. As in Theorem \ref{T1},
we can also assume that $\left\vert x\right\vert ,\left\vert y\right\vert $
are large enough.

$\left( i\right) $ Let all the ends be subcritical. Then we have again (\ref%
{D x+y}). Using $\sqrt{t}>\left\vert x\right\vert $ and (\ref{equivalence of
e}), we can replace $e^{-b\frac{\left\vert x\right\vert ^{2}+\left\vert
y\right\vert ^{2}}{t}}$ in the right hand side of (\ref{D x+y}) by $e^{-b%
\frac{d^{2}\left( x,y\right) }{t}}$. Using further (\ref{Vixo}), we can
replace $V_{i}\left( x,\sqrt{t}\right) $ by $V_{i}\left( \sqrt{t}\right) $
and, hence, by $V\left( \sqrt{t}\right) $, which yields%
\begin{equation*}
p\left( t,x,y\right) \asymp \frac{C}{V(\sqrt{t})}\left(
D(x,t)D(y,t)+1+D(x,t)+D(y,t)\right) e^{-b\frac{d^{2}\left( x,y\right) }{t}},
\end{equation*}%
and which implies (\ref{Vi}) since $D(x,t)$, $D(y,t)$ are bounded.

$\left( ii\right) $ Let at least one of the ends be critical. Then by $%
V_{i}\left( r\right) \approx V\left( r\right) $, the end $M_{i}$ has to be
critical, too. As in the case $\left( ii\right) _{2}$ of the proof of
Theorem \ref{T2}$\left( a\right) $, we obtain again (\ref{same end ii2}),
where by (\ref{equivalence of e}) we can replace $e^{-b\frac{\left\vert
x\right\vert ^{2}+\left\vert y\right\vert ^{2}}{t}}$ in the right hand side
of (\ref{same end ii2}) by $e^{-b\frac{d^{2}}{t}}.$ Using further (\ref{Vixo}%
), we replace $V_{i}\left( x,\sqrt{t}\right) $ by $V_{i}\left( \sqrt{t}%
\right) \approx V\left( \sqrt{t}\right) \approx t$, which yields 
\begin{eqnarray*}
p(t,x,y) &\asymp &\frac{C}{t}\left[ W(x,t)W(y,t)+U(x,t)U(y,t)+W(x,t)U\left(
y,t\right) +W(y,t)U(x,t)\right] e^{-b\frac{d^{2}}{t}} \\
&=&\frac{C}{t}\left\{ W(x,t)+U(x,t)\right\} \left\{ W(y,t)+U(y,t)\right\}
e^{-b\frac{d^{2}}{t}}.
\end{eqnarray*}%
Using (\ref{U+W}), we conclude (\ref{Vi}).
\end{proof}

\begin{proof}[Proof of Theorem \protect\ref{T3}]
As in Theorem \ref{T1}, we can assume that $\left\vert x\right\vert
,\left\vert y\right\vert $ are large enough. Since $\sqrt{t}\geq \min \left(
\left\vert x\right\vert ,\left\vert y\right\vert \right) $ and the both
estimates (\ref{T3i}) and (\ref{T3ii}) are symmetric in $x$, $y$, so we can
assume without loss of generality that $\sqrt{t}\geq \left\vert x\right\vert 
$. Then, by Lemma \ref{lemma e}, the function $V_{i}\left( x,\sqrt{t}\right) 
$ in the estimates (\ref{Dirichlet subcritical}) and (\ref{Dirichlet
critical}) can be replaced by $V_{i}\left( \sqrt{t}\right) $.

$\left( i\right) $ Assume that all ends are subcritical. Applying (\ref%
{equivalence of e}) to (\ref{D x+y}) and observing that the function $D$ is
bounded, we obtain (\ref{T3i}).

$\left( ii\right) $ Let at least one of the ends be critical. Since $M_{i}$
is subcritical, substituting (\ref{Dirichlet subcritical}), (\ref{Ppm}), (%
\ref{central integral critical}), (\ref{hitting subcritical}) and (\ref%
{hitting derivative subcritical}) into (\ref{general full estimate}), we
obtain 
\begin{align*}
p(t,x,y)\asymp & C\frac{D(x,t)D(y,t)}{V_{i}(\sqrt{t})}e^{-b\frac{d^{2}}{t}}+%
\frac{C}{t}e^{-b\frac{\left\vert x\right\vert ^{2}+\left\vert y\right\vert
^{2}}{t}}  \notag \\
& +C\frac{\log t}{t}\left( D(x,t)+D(y,t)\right) e^{-b\frac{\left\vert
x\right\vert ^{2}+\left\vert y\right\vert ^{2}}{t}},
\end{align*}%
which together with (\ref{equivalence of e}) implies (\ref{T3ii}).
\end{proof}

\begin{acknowledgement}
This work was completed during the stay of the first and second authors in
the Institute of Mathematical Sciences of Chinese University of Hong Kong.
The authors are grateful to CUHK for the hospitality and support. The
authors would like to thank Professor Yuji Kasahara for sending his preprint.
\end{acknowledgement}

\end{document}

%% file: macro.tex
\usepackage{hyperref}

\def\RM{\rm}
\def\limfunc#1{\mathop{\mathrm{#1}}}
\def\func#1{\mathop{\mathrm{#1}}\nolimits}

\def\subjclass#1{\par{\footnotesize {\bf 2010 Mathematics Subject Classification:} #1}}
\def\keywords#1{\par{\footnotesize {\bf Keywords:} #1}}

\def\Qlb#1{#1}

\def\Qcb#1{#1}

\def\FRAME#1#2#3#4#5#6#7#8
{
 \begin{figure}[ptbh]
 \begin{center}
 \includegraphics[height=#3]{#7}
 \caption{#5}
 \label{#6}
 \end{center}
 \end{figure}
}